\documentclass[11pt, hidelinks]{article}
\usepackage[margin =1in]{geometry}
\usepackage{amssymb,amsthm,amsmath}
\usepackage{enumerate}
\usepackage{hyperref}
\usepackage{cite}
\usepackage[capitalize]{cleveref}
\usepackage[shortlabels]{enumitem}
\usepackage{color}
\usepackage{cancel}
 \usepackage{tikz}
\usepackage{tikz-3dplot}
\usepackage{pgf}
\usepackage{svg}
\usepackage{pgfplots}
\usepackage{import}
\usepackage{svg}
\usepackage{comment}
\usepackage[utf8]{inputenc} 
\usepackage[T1]{fontenc}    
\usepackage{hyperref}       
\usepackage{url}            
\usepackage{booktabs}       
\usepackage{amsfonts}       
\usepackage{nicefrac}       
\usepackage{microtype}
\usepackage{multirow}
\usepackage{xfrac}
\usepackage{todonotes}
\usepackage{titlesec}
\usepackage{caption}
\usepackage{subcaption}

\usepackage{algorithm}
\usepackage{algpseudocode}
\usepackage{array}

\algtext*{EndFor}
\algtext*{EndIf}
\algtext*{EndWhile}
\algtext*{EndProcedure}
\algrenewtext{For}[1]{\textbf{for} #1}
\algrenewtext{While}[1]{\textbf{while} #1}
\algrenewtext{If}[1]{\textbf{if} #1}

\algnewcommand{\Input}{\item[\textbf{Input:}]}
\algnewcommand{\Initialize}{\item[\textbf{Initialize:}]}

\usepackage{varwidth}

\titleformat{\subsubsection}[runin]
{\normalfont\normalsize\bfseries}{\thesubsubsection}{1em}{}

\usepackage{graphicx}

\usepackage{appendix}
\numberwithin{equation}{section}

\newcommand{\spann}[1]{\mathrm{span}\left\{{#1}\right\}}

\newcommand{\dom}[1]{\mathrm{dom }(#1)}

\newcommand{\argmin}{\operatornamewithlimits{argmin}}

\newcommand{\Rext}{(-\infty, \infty]}

\theoremstyle{thmstyletwo}%
\newtheorem{thm}{Theorem}[section]
\newtheorem{theorem}[thm]{Theorem}
\newtheorem{definition}[thm]{Definition}
\newtheorem{proposition}[thm]{Proposition}

\newtheorem{lemma}[thm]{Lemma}

\newtheorem{corollary}[thm]{Corollary}

{\theoremstyle{plain}
}

\newtheorem{conjecture}[thm]{Conjecture}

\numberwithin{equation}{section}
\crefname{claim}{claim}{claims}
\Crefname{claim}{Claim}{Claims}
\crefname{lem}{lemma}{lemmas}
\Crefname{lem}{Lemma}{Lemmas}
\crefname{algorithm}{algorithm}{algorithms}
\Crefname{algorithm}{Algorithm}{Algorithms}

\newcommand{\eps}{\varepsilon}

\newcommand{\conv}[1]{\textrm{conv}\left\{{#1}\right\}}

\newcommand{\R}{\mathbb{R}}
\newcommand{\N}{\mathbb{N}}

\newcommand{\Ical}{\mathcal{I}}
\newcommand{\Holder}{H\"older }
\newcommand{\inner}[2]{\left\langle {#1}, {#2} \right\rangle}
\newcommand{\Lleft}{L^{\leftarrow}}

\usepackage{mathtools}

\usepackage[T1]{fontenc}
\usepackage{lmodern}
\newcommand{\pAlg}{p^\mathtt{alg}}
\newcommand{\dAlg}{d^\mathtt{alg}}
\newcommand{\pPEP}[1]{p_{#1}}
\newcommand{\dPEP}[1]{d_{#1}}
\newcommand{\Algclass}[1]{\mathtt{A}_{\mathrm{{#1}}}}

\newlength{\solutionindent}
\newlength{\questionboxrule}
\begin{document}
    \title{Inexactly Smooth Performance Estimation\\ and New Optimized Gradient Methods}

	 \author{Aaron Zoll\footnote{Johns Hopkins University, Department of Applied Mathematics and Statistics, \url{azoll1@jhu.edu}} \qquad Benjamin Grimmer\footnote{Johns Hopkins University, Department of Applied Mathematics and Statistics, \url{grimmer@jhu.edu}}}

	\date{}
	\maketitle

	\begin{abstract}
        We consider a general class of ``inexactly smooth'' convex functions, providing a universal model capturing as special cases $L$-smooth, $M$-Lipschitz,  and H\"older smooth functions, and any combination thereof. Such functions possess a calculus closely following that of smooth functions. Our main results provide inexactly smooth functions with interpolation theorems that are necessary and sufficient up to modest universal constants. These enable analysis of first-order methods for any inexactly smooth convex problem class via solving convex Performance Estimation Problems (PEPs). Further, these enable the extension of Drori and Taylor~\cite{constructive_approach}'s constructive approach to algorithm design. From this, we derive an exactly minimax optimal method for $(\beta,0)$-H\"older smooth problems, methods with the best-known convergence guarantees up to constants for any $(\beta,p)$-H\"older smooth convex minimization, and a new universal fast backtracking method for any inexactly smooth convex problem.
	\end{abstract}

\section{Introduction}
First-order methods have found success in a wide range of convex optimization tasks, ranging from smooth minimization to nonsmooth minimization. First-order methods are particularly effective at modern, large-scale unconstrained optimization problems of the form
$$ f_\star = \min_{x\in\R^d} f(x) $$
where computations of (sub)gradients remain tractable while other oracle models become expensive. Since the 1970s, algorithms tailored to the structure of the given problem class (smooth, nonsmooth, etc.) have been developed with highly optimized performance, often being minimax optimal. As classic examples, see~\cite{nesterov2004introductory, nemirovskilowerbounds}.

Alas, these optimized methods are typically only guaranteed to be performant on their problem class, arising from tailored analysis techniques. Here, our goal is to develop analytic tools, in particular ``Performance Estimation Problems,'' uniformly applicable to algorithm design over a wide range of problem classes.

To this end, we are motivated by the universal methods and analysis of~\cite{devolder_inexactly_smooth} and~\cite{Nesterov_2014}. Therein, they design first-order methods given access to a first-order oracle returning $(f(x),g) $ with $g \in\partial f(x)$ for any query $x$. Here $\partial f(x)=\{g : f(y) \geq f(x)+\langle g, y-x\rangle\ \forall y\}$ denotes the convex subdifferential. In particular, they considered any convex problem with \Holder continuous subgradients, referred to here as $(\beta,p)$-\Holder smoothness,
\begin{equation}\label{eq:Holder-def}
    \|g_x - g_y\|\leq \beta\|x-y\|^p \qquad \forall x,y\in\R^d,\ g_x\in\partial f(x),\ g_y\in\partial f(y)
\end{equation}
with $\beta>0$ and $p\in [0,1]$.
Nesterov's Universal Fast Gradient Method~\cite{Nesterov_2014} ensures
\begin{equation}\label{eq:Optimal-Holder-rate}
    f(x_N) - f_\star \leq \frac{2^{1+2p}\beta \|x_0-x_\star\|^{1+p}}{N^{\frac{1+3p}{2}}}
\end{equation}
universally, which is within a constant factor of the minimax optimal rate~\cite{nemierovski_optimal_HS} for each $p\in[0,1]$. When $p=1$, \Holder smoothness corresponds to the classical smoothness assumption of $\beta=L$-Lipschitz gradient. When $p=0$, this class corresponds to functions with subgradients differing by at most $\beta$. This is closely related, but distinct, from the classical nonsmooth assumption of having $M = \beta$-Lipschitz function values. Considering $p\in (0,1)$ interpolates between these settings.

The key analysis tool leveraged by~\cite{devolder_inexactly_smooth} and~\cite{Nesterov_2014} is that \Holder smoothness~\eqref{eq:Holder-def} implies the inexact quadratic upper bound for 
all $x,y \in \R^d,\ g_x\in\partial f(x)$ and any $\delta \geq 0$
\begin{equation}\label{eq:Holder-consequence}
    f(y) \leq f(x) + \langle g_x, y-x\rangle + \frac{\left(\frac{1-p}{1+p}\cdot\frac{1}{2\delta}\right)^\frac{1-p}{1+p}\beta^{\frac{2}{1+p}}}{2}\|y-x\|^2 + \delta.
\end{equation}
This result allows one to approach any \Holder smooth setting with similar methods and proof techniques to the smooth setting ($p=1$) where this bound holds without the additive error $\delta$. Such inexact bounds provide a means to universal designs.

Generalizing this key analysis tool, here we propose a family of ``inexactly smooth'' functions. We say a convex function is $L(\cdot)$-inexactly smooth for some function $L\colon [0,\infty) \rightarrow (0, \infty]$, if for any $x, y \in \R^d$,  $g_x\in\partial f(x)$, and $\delta\geq 0$,
\begin{equation}\label{def: inexactly smooth}
    f(y) \leq f(x) + \inner{g_x}{y-x} + \frac{L(\delta)}{2}\|y-x\|^2 + \delta.
\end{equation}
If $L(\delta) = \left(\frac{1-p}{1+p}\cdot\frac{1}{2\delta}\right)^\frac{1-p}{1+p}\beta^{\frac{2}{1+p}}$, this family contains all $(\beta,p)$-\Holder smooth functions. This can additionally model functions generated by sums of smooth and nonsmooth components, which have also received notable theoretical interest recently~\cite{grimmer2023heterogeneous, wu2026universalparameterfreegradientsliding, guigues2025universalsubgradientproximalbundle, Jelena_fine_grained}. Given a sum $f=f_0+f_1$, with $f_0$ having subgradients differ by at most $\beta_0$ and $f_1$ having $\beta_1$-Lipschitz gradient, inexact smoothness holds with $L(\delta)=\beta_1+\beta_0^2/(2\delta)$. Section~\ref{sec:Interpolation} discusses more properties and examples within this model.

In this paper, we develop a principled approach to the design and analysis of algorithms for inexactly smooth convex minimization. The foundational works of~\cite{drori2014performance} and~\cite{taylor2017smooth,taylor2017composite} introduced Performance Estimation Problems (PEPs) as a principled way to design and analyze first-order methods over traditional problem classes. See PEPit~\cite{goujaud2024pepit} and the many examples therein. PEPs are mathematical programs (often semidefinite programs) that compute a worst-case problem instance from a given class (e.g., $L$-smooth convex problems) for a given algorithm. Dually, PEPs provide a best possible convergence proof for a given algorithm using a given set of inequalities. Agreement of these primal and dual perspectives relies on related ``Interpolation Theorems.'' 
Kamri~\cite[Chapter 4]{kamri2025} provided a numerical study of some PEPs for \Holder smooth problems, without the analytical support of such theorems.

\paragraph{Our Contributions.}
We show that interpolation, PEP, and the constructive approach to algorithm design of~\cite{constructive_approach} all extend to inexactly smooth functions. Namely,
\begin{itemize}
    \item \textbf{Calculus of Inexactly Smooth Functions} In Section~\ref{sec:inexactly smooth functions}, we derive characterizations of inexactly smooth functions and their conjugates, as well as associated calculus rules.
    \item  \textbf{Interpolation Theory} Section~\ref{sec:Interpolation} begins by presenting necessary and sufficient conditions for observations to be interpolable by an inexactly smooth function. Theorem~\ref{thm:interpolation-general} gives our most general interpolation theory, which is tight up to a small universal constant.  Theorem~\ref{thm:interpolation-p0} establishes that these conditions are exact for $(\beta,0)$-\Holder smooth functions. 
    \item  \textbf{Inexactly Smooth Performance Estimation} We conclude Section~\ref{sec:Interpolation} by using our interpolation theory to define a convex formulation for analyzing the worst-case performance of an algorithm. Note, unlike most existing PEP work, these convex programs may not be semidefinite programs. Regardless, Theorem~\ref{thm:strong duality} establishes a strong duality result.
   \item \textbf{Algorithm Design} Section~\ref{sec:algo design} leverages our interpolation theory to construct optimized algorithms. \Cref{alg:SSEP_BSD} is minimax optimal among $(\beta, 0)$-\Holder smooth convex functions, Algorithm~\ref{alg:IS-OGM} is an asymptotically optimized gradient method for general \Holder smooth convex minimization, and Algorithm~\ref{alg:UOBL} is a universal method for any inexactly smooth convex problem requiring only input of a target accuracy $\eps$ and an iteration budget $N$. The associated convergence guarantees are given in Theorems~\ref{thm:minimax_BSFOM}, \ref{thm:UOGM-k/delta-q}, and \ref{thm:UOBL convergence}, respectively. 
\end{itemize}
\noindent Note the inexactness considered here~\eqref{def: inexactly smooth} is inexactness in attainment of smoothness-type quadratic upper bounds. This is distinct from inexactness due to gradient calculations, which have separately been studied in the PEP literature~\cite{liu2025nonasymptoticanalysisacceleratedmethods, deklerk2017exactlinesearch, deklerk2020inexact, gannot2022frequency}.

 \section{Preliminaries and Inexactly Smooth Function Calculus} \label{sec:inexactly smooth functions}

We begin with preliminaries and then discuss \Holder smooth functions as our motivating instance of inexactly smooth functions in Section~\ref{subsec:prelim}. Then Section~\ref{subsec:calculus} develops a general calculus for $L(\cdot)$-inexactly smooth functions which may be of independent interest.

Throughout, we consider closed, convex, proper functions $f: \R^d \to \Rext$. Let $\inner{\cdot}{\cdot}$ denote the standard Euclidean inner product with the associated two-norm $\|\cdot\|$. We denote the domain of $f$ by $\dom{f}$ and its subdifferential as previously introduced by $\partial f(x)$. We denote $\mathrm{epi}(f) := \{(x,t) \in \R^d \times \R: f(x) \leq t\}$ as the epigraph of $f$, for which the function $f$ is convex if and only if the set $\mathrm{epi}(f)$ is convex. Similarly, we denote $\mathrm{hypo}(f) :=  \{(x,t) \in \R^d \times \R: f(x) \geq t\}$ as the hypograph. The Fenchel conjugate of $f$ is $f^*(g)=\sup_{x \in \R^d} \{\inner{g}{x}-f(x)\}$, which satisfies the Fenchel-Young inequality, which holds with equality exactly when $g\in\partial f(x)$,
\begin{equation} \label{eq:Fenchel-Young}
    f(x) + f^*(g) \geq \langle g,x\rangle \qquad \forall x,g\in\mathbb{R}^d.
\end{equation}

\subsection{\Holder Smooth Functions} \label{subsec:prelim}
A core motivation for our work on general inexactly smooth functions is the special case of $(\beta,p)$-\Holder smooth functions (i.e., having \Holder continuous (sub)gradient as defined in~\eqref{eq:Holder-def}).
The extreme cases of $p=1$ and $p=0$ are particularly fundamental. When $p=1$, this is exactly $f$ having $\beta$-Lipschitz gradient and corresponds to $L(\delta)=\beta$-inexact smoothness. When $p=0$, this bounds the difference between two subgradients by $\beta$ and corresponds to $L(\delta)=\frac{\beta^2}{2\delta}$.

\Holder smoothness for any $p\in [0,1]$ was related to $L(\delta)$-inexact smoothness by~\cite{devolder_inexactly_smooth}. Below we formalize this, where the coefficient $\left(\frac{p+1}{2p}\right)^p$ is defined at $p=0$ by its limiting value of $1$ as $p\rightarrow 0$. By convention, we take $1/\infty = 0$ and $1/0 = \infty$.

\begin{proposition}\label{prop:Equiv Char of HS functions}
    Let $f : \R^d \to \Rext$ be closed convex and proper, and $L(\delta) = \left(\frac{1-p}{1+p}\frac{1}{2\delta}\right)^\frac{1-p}{1+p} \beta^{\frac{2}{1+p}}$. Then the conditions
    \begin{enumerate}
        \item $\|g_x-g_y\| \leq \beta \|x-y\|^p$, $\quad \forall$ $x, y \in \dom{f}$ and $g_x \in \partial f(x),  g_y \in \partial f(y) $,
      \item $f(y) \leq f(x) + \inner{g_x}{y-x}+\frac{L(\delta)}{2}\|x-y\|^{2}+\delta$, $\quad \forall$ $x, y \in \dom{f}$, $g_x \in \partial f(x)$, and $\delta \geq 0$,
        \item $f(y) \geq f(x) + \inner{g_x}{y-x}+\frac{1}{2L(\delta)}\|g_x-g_y\|^{2} - \delta$, $\quad \forall$ $x, y \in \dom{f}$, $g_x \in \partial f(x),  g_y \in \partial f(y) $, and $\delta \geq 0$,
        \item $\|g_x-g_y\| \leq \beta \left(\frac{p+1}{2p}\right)^p\|x-y\|^p$, $\quad \forall$ $x, y \in \dom{f}$ and $g_x \in \partial f(x),  g_y \in \partial f(y) .$
    \end{enumerate}
    are related by the implications $1. \Longrightarrow 2. \Longleftrightarrow 3. \Longrightarrow 4.$
\end{proposition}
\begin{proof}
    The proof of~\cite[Lemma 1]{li2024simpleuniformlyoptimalmethod} proves $(1. \Longrightarrow 2. \Longrightarrow 3. \Longrightarrow 4.)$ for all $p \in (0,1)$. Considering limits as $p \to 0$ and $p \to 1$ extends this to $p\in [0,1]$. In \cref{lemma:L(delta) conjugate equivalence}, we will provide the final needed implication, establishing $(2 \Longleftrightarrow 3.)$.
 \end{proof}

Note that the above implications are equivalent up to this factor of $\left(\frac{p+1}{2p}\right)^p \leq 1.263$. This factor is tight. Hence, a small constant factor gap is fundamental when approximating the family of \Holder smooth functions via inexact smoothness. Tightness is demonstrated by the following function that satisfies conditions 2.~through 4., but holds with equality in 4.~when $x=0$ and $y=1$,
\begin{equation}\label{ex:tightness-Holder-condition}f_{\beta, p}(x) := \begin{cases}
    0 & x < 0\\
    \frac{1}{2}\beta\left(\frac{p+1}{2p}\right)^{p}x^2 & 0 \leq x \leq 1\\
    \beta\left(\frac{p+1}{2p}\right)^{p}\left(x-\frac{1}{2}\right) & x > 1.
\end{cases}\end{equation}

\subsection{Calculus of $L(\cdot)$-inexactly smooth functions}\label{subsec:calculus}

Here, we develop a calculus for $L(\cdot)$-inexactly smooth functions previously defined in \eqref{def: inexactly smooth}.
Complementing this, we say a convex function $f$ is $\mu(\cdot)$-inexactly strongly convex if for all $\delta\geq 0$,
\begin{equation}\label{def: inexactly mu strong convexity}
    f(y) \geq f(x) + \inner{g_x}{y-x} + \frac{\mu(\delta)}{2}\|x-y\|^2 - \delta, \quad \forall x, y \in \mathbb{R}^d,\ g_x\in\partial f(x).
\end{equation}

Our development of characterizations of these inexact quantities closely follows the classical developments of smoothness and strong convexity. Consequently, we defer proofs to \cref{sec: appendix L(delta) calc} when they are analogous.

\noindent {\bf Assumptions on $L(\delta)$:} Inexact smoothness functions $L \colon [0,\infty) \to (0, \infty]$ are finite for $\delta > 0$, convex, lower semicontinuous, nonincreasing, and have $-1/L(\cdot)$ convex. 

\begin{lemma}\label{lemma:equivalences-of-convexity-assumptions}
    For any $L \colon [0,\infty) \to (0, \infty]$ that is convex, lower semicontinuous, and nonincreasing, one has that
    $$ -1/L(\delta)\text{ is convex} \quad \iff \quad s \Lleft(s) \text{ is convex},$$
    where we define the left inverse $ \Lleft(s) := \inf\{\delta  \geq 0 : L(\delta) \leq s\}. $
\end{lemma}

\begin{proof}
Consider the set $\mathcal S:=\left\{\left(\frac{1}{\alpha},\frac{\delta}{\alpha}\right): (\delta,\alpha)\in \mathrm{hypo}(1/L), \, \alpha > 0\right\}$. We claim that $\mathcal S$ is the epigraph of $s\mapsto s\Lleft(s)$. The condition $(\delta,\alpha)\in\mathrm{hypo}(1/L)$ means $\alpha\le 1/L(\delta)$, equivalently $L(\delta)\le 1/\alpha$. Considering the definition of $\mathcal{S}$ and writing $s:=1/\alpha$ and $z:=\delta/\alpha=\delta s$, the previous inequality states there exists $\delta \geq 0$ such that $L(\delta)\le s$ and $z=\delta s$. Since $L(\cdot)$ is nonincreasing and $1/L(\cdot)$ is nondecreasing, we may rewrite
$$
\mathcal S=\{(s,z):\ \exists \delta>0\ \text{s.t.}\ L(\delta)\le s\ \text{and}\ z\ge \delta s\}.
$$
Fix some $s>0$. Among all $\delta \geq0$ with $L(\delta)\le s$, the smallest possible value of $\delta s$ is
$$
s\cdot\inf\{\delta \geq 0:\ L(\delta)\le s\}=s\,\Lleft(s).
$$
Hence $\mathcal S$ is the epigraph of the function $s\mapsto s\Lleft(s)$.

The following equivalences hold between convexity of functions and convexity of their associated epigraphs/hypographs $$-1/L(\delta) \Longleftrightarrow \mathrm{hypo}(1/L(\delta)) \Longleftrightarrow \mathrm{epi}(s \Lleft(s)) \Longleftrightarrow s \Lleft(s),$$
where the first and last implications note convexity and concavity are equivalent to convexity of epigraphs and hypographs, and the middle  holds as $\mathcal{S}$ is a perspective transformation of the hypograph intersected with a halfplane~\cite[Section 2.3.3]{boyd2004convex}. 
\end{proof}

First, we note that one can relax an inexact smoothness function $L(\cdot)$ (under our assumptions) into the sum of two inexact smoothness functions previously seen for $p=1$ and $p=0$ \Holder smoothness, $L(\delta) \leq \beta_1 + \beta_0^2/(2\delta)$. Therefore, these extremal cases of \Holder smoothness are also extremal among inexactly smooth functions.  
\begin{lemma}\label{smooth plus nonsmooth bound on L}
    If $L(\cdot)$ satisfies our assumptions, then for any $a > 0$, $L(\delta) \leq \frac{aL\left(a\right)}{\delta}+L\left(a\right)$.
\end{lemma}
\begin{proof}
   By concavity of $1/L(\cdot)$ it holds that for any $a > 0$ and $t \in [0,1]$ that $$\frac{1}{L(ta)} = \frac{1}{L(ta + (1-t)0)} \geq \frac{t}{L(a)}+\frac{1-t}{\lim_{\delta \to 0^+}L(\delta)} \geq \frac{t}{L(a)},$$ where the last inequality comes from the positivity of $L(\cdot)$. Letting $\delta = ta$, this implies that $L(\delta) \leq \frac{aL(a)}{\delta}$ for any $\delta\in[0,a]$. Then by monotonicity of $L(\cdot)$, $L(\delta) \leq L(a)$ holds for all $\delta \geq a$. Summing the bounds from these two regimes gives the claim.
\end{proof}

Next, we derive useful equivalent characterizations of $L(\cdot)$-inexact smoothness. The following lemma shows that among convex functions $f$, inexact smoothness can equally be viewed as satisfying a cocoercivity-type condition or as the conjugate $f^*$ being inexactly strongly convex with $\mu(\cdot) = 1/L(\cdot)$. These follow similarly to the standard textbook fact for (non-inexactly) smooth convex functions.

\begin{lemma}\label{lemma:L(delta) conjugate equivalence}
    Let $f$ be closed, convex, and proper. The following are equivalent:
    \begin{enumerate}        
        \item $f(y) \geq f(x) + \inner{g_x}{y-x}+\frac{1}{2L(\delta)}\|g_y-g_x\|^2 - \delta, \quad \forall \delta \geq 0, \ \forall x,y \text{ with } g_x \in \partial f(x),\ g_y\in \partial f(y)$,
        \item $f$ is $L(\cdot)$-inexactly smooth,
        \item $f^*$ is $1/L(\cdot)$-inexactly strongly convex. 
    \end{enumerate}
\end{lemma}

The following three lemmas provide collections of calculus rules for constructing inexactly smooth functions and quantifying their inexact smoothness $L(\cdot)$. As with the above result, their proofs are analogous to the standard smooth and convex facts and hence, deferred to \cref{sec: appendix L(delta) calc}. Note, by considering conjugate functions and \cref{lemma:L(delta) conjugate equivalence}, equivalent calculus rules for inexactly strongly convex functions could be derived directly from these.
\begin{lemma}\label{L(delta) scaling calculus}
    Consider any convex, $L(\cdot)$-inexactly smooth function $f$. Then, 
    \begin{enumerate}
        \item $\alpha f(x)$ is $\alpha L(\cdot/\alpha)$-inexactly smooth for any $\alpha > 0$,
        \item $f(Ax - b)$ is $\|A\|_{\mathrm{op}}^2L(\cdot)$-inexactly smooth,
        \item Suppose $F(x) = \inf_z f(x,z)$ is closed and proper with attainment for each $x$, then $F(x)$ is $L(\cdot)$-inexactly smooth.
    \end{enumerate}
\end{lemma}

\begin{lemma}\label{lemma:L(delta) sums and max calculus}
    Consider any closed, convex, and proper $L_i(\cdot)$-inexactly smooth functions $f_i$ for $i=1,\dots,m$. Then, \begin{enumerate}
        \item $\sum_{i=1}^m f_i$ is $\left(\inf_{\substack{\sum_{i=1}^m \delta_i = \delta\\ \delta_i \geq 0}} \sum_{i=1}^m L_i(\delta_i)\right)$-inexactly smooth,
        \item $(\max_i f_i^*)^*$ is $\max_i L_i(\delta)$-inexactly smooth,
        \item $\left(\sum_{i=1}^m f_i^*\right)^*$ is $\left(\sup_{\substack{\sum_{i=1}^m \delta_i = \delta\\ \delta_i \geq 0}} \sum_{i=1}^m \frac{1}{L_i(\delta_i)}\right)^{-1}$-inexactly smooth.
    \end{enumerate}
\end{lemma}

\begin{lemma}\label{composing with abs to norm}
    Consider any convex function $h : \R \to \R $ such that $h(|t|)$ is $L(\cdot)$-inexactly smooth. Then $f(x) = h(\|x\|)$ is $L(\cdot)$-inexactly smooth.  
\end{lemma}

Rather than retaining a family of inequalities for each $\delta\geq 0$, one may consider the strengthened inequality given by optimizing over all $\delta\geq 0$. To this end, we define
\begin{equation}\label{def:theta}
    \phi_L(s,\delta) := \frac{L(\delta)}{2}s^2 + \delta, \quad \theta_L(s
    ) := \inf_{\delta \geq 0} \phi_L(s,\delta).
\end{equation}
Below notes properties of the joint function $\phi_L(s,\delta)$ and its partial minimization $\theta_L$.

\begin{lemma}\label{lemma:phi_calc}
    Suppose $L(\cdot)$ satisfies our assumptions. Then $\phi_L(s,\delta)$ as defined in \eqref{def:theta} is jointly convex in $s$ and $\delta$. When attained, $\{\delta_\star(s)\} := \argmin_{\delta \geq 0} \phi_L(s,\delta)$ is a singleton, and if $s \ne 0$, then $L(\delta_\star(s)) < \infty$.
\end{lemma}
\begin{proof}
   Consider the jointly convex function $g(u,v) = u^2/v$, decreasing in $v > 0$. Since $1/L(\delta)$ is concave, the composition $\phi_L(s,\delta) = g(s, 2/L(\delta)) + \delta$ is jointly convex. 
   
   We now show that $\argmin_{\delta \geq 0} \phi_L(s,\delta)$ is a singleton for all $s$. If $ s = 0$, this is trivial. Otherwise, without loss of generality, assume $s > 0$ (since $\phi_L(\cdot,\delta)$ is even) and suppose the convex level set, $\argmin_{\delta \geq 0} \phi_L(s,\delta) = [\delta_1, \delta_2]$. Because $\phi_L(s,\cdot)$ is convex, it must hold that $\phi_L(s,\delta)  = \theta_L(s)$ is constant on $[\delta_1, \delta_2]$. However, this implies that $$L(\delta) = \frac{2}{s^2}(\theta_L(s) - \delta), \quad \delta \in [\delta_1, \delta_2]$$ for fixed $s > 0$ and $\theta_L(s) > 0$. Since $L(\delta) > 0$ by assumption, $-1/L(\cdot)$ is strictly concave on this interval. This can only be compatible with the assumption that $-1/L(\cdot)$ is convex when $\delta_1 = \delta_2$. Therefore, we have $\argmin_{\delta \geq 0} \phi_L(s,\delta)$ as a singleton when attained. Finally, when $s \ne 0$, since $L(\cdot)$ is finite for all $\delta > 0$, it holds that
   $$L(\delta_\star(s)) \leq \frac{\min_{\delta \geq 0} \phi_L(s,\delta)}{s^2/2} \leq \frac{\phi_L(s,1)}{s^2/2} < \infty.  $$
\end{proof}
\begin{lemma}\label{lemma:theta_L-and-conjugate_properties}
    Suppose $L(\cdot)$ satisfies our assumptions. Then $$\theta_L^*(u) =\sup_{\delta \geq 0}\left\{\frac{1}{2L(\delta)}u^2-\delta\right\},$$ and both $\theta_L$ and $\theta_L^*$ are closed, convex, proper, and even.  Finally, $\theta_L$ vanishes only at $0$, and for all $s \ne 0$, the subdifferential is a singleton.
\end{lemma}
\begin{proof}
First, we verify that $\theta_L$ is closed, convex, and proper. Closedness and properness follow immediately from the assumed structure of $L(\cdot)$. Convexity follows since $\theta_L(s)$ is the partial minimization of $\phi_L(s,\delta)$, see~\cite[Theorem 5.3]{rockafeller}.

Closedness and convexity of $\theta_L^*$ follow as it is a conjugate. Properness of $\theta_L^*$ follows from the properness of $\theta_L$. The claimed formula for this conjugate follows as
\begin{align*}
  \sup_s \sup_{\delta \geq 0}\left\{ us - \delta - \frac{L(\delta)}{2}s^2\right\}= \sup_{\delta \geq 0} \sup_s \left\{ us - \delta - \frac{L(\delta)}{2}s^2\right\}= \sup_{\delta \geq 0} \left\{\frac{1}{2L(\delta)}u^2-\delta\right\}.
\end{align*}

Notice that both functions are even as they only depend on their input through a squared term, so $\theta_L(s) = \theta_L(-s)$ and $\theta_L^*(u)=\theta_L^*(-u)$.

As $L(\delta)$ is finite for $\delta > 0$, it holds that $\theta_L(0)=\inf_{\delta \geq 0} \delta = 0$. If $s \ne 0$, we can  show that $\theta_L(s) > 0$ by providing a uniform lower bound on $\phi_L(s,\delta) = \tfrac{L(\delta)}{2} s^2 + \delta$ for all $\delta \geq 0$. First, for $\delta \leq 1$, by monotonicity of $L(\cdot)$, we can bound $\phi_L(s,\delta) \geq \frac{L(1)}{2}s^2 > 0$.
Now suppose  $\delta > 1$. By concavity and monotonicity of $1/L(\cdot)$, there exists some $a \geq 0$ such that $L(\delta) \geq \frac{1}{\frac{1}{L(1)}+a(\delta-1)}$. By the arithmetic-geometric mean inequality, $$\frac{L(\delta)}{2}s^2+\delta \geq |s|\sqrt{\frac{2L(1)\delta}{1+aL(1)(\delta-1)}} \geq |s|\sqrt{\min\{2L(1), 2/a\}} > 0,$$
where the first bound substitutes the bound derived from the concavity of $1/L(\cdot)$ and the second bound considers whether $aL(1) \leq 1$ or $aL(1) > 1$. Therefore,
$$\theta_L(s) = \inf_{\delta\geq 0} \phi_L(s,\delta) \geq \min\left\{\frac{L(1)}{2}s^2, |s|\sqrt{2L(1)}, |s| \sqrt{2/a}\right\} > 0. $$

Lastly, since $\phi_L(s,\delta)$ is closed, convex, and proper, \cite[Theorem 3.101]{Hoheisel2019} ensures \begin{align*}\partial \theta_L(s) &= \{u : (u,0) \in \partial \phi_L(s,\delta_\star(s)), \ \delta_\star(s) \in \argmin_{\delta \geq 0} \phi_L(s, \delta)\}.\end{align*}
For $s \ne 0$, \cref{lemma:phi_calc} shows $\argmin_{\delta \geq 0} \phi_L(s,\delta) = \{\delta_\star(s)\}$ is unique and $L(\delta_\star(s)) < \infty$. Coupled with $\phi_L(\cdot,\delta_\star(s))$ being differentiable in $s$ at $(s, \delta_\star(s))$ since $L(\delta_\star(s)) < \infty$, it holds that $\partial \theta_L(s)$ is a singleton: non-empty by convexity of $\theta_L$, and every element of $\partial\phi_L(s, \delta_\star(s))$ has first coordinate equal to $L(\delta_\star(s))s$, as outlined above. 
\end{proof}


\section{Inexactly Smooth Performance Estimation}\label{sec:Interpolation}
This section states and proves our main results on the interpolation theory for $L(\cdot)$-inexactly smooth functions. 

\begin{definition}\label{def:L()-interpolation}
    Consider observations $\mathcal{H} = \{(x_i,f_i,g_i)\}_{i\in\Ical}$ with $x_i, g_i \in \R^d$ and $f_i \in \R$ for some index set $\Ical$. We say $\mathcal{H}$ is ${L(\cdot)}$-interpolable if there exists a convex, $L(\cdot)$-inexactly smooth function $f$ such that $f(x_i) = f_i$ and $g_i \in \partial f(x_i)$ for all $i \in \Ical$. 
\end{definition}

Our main goal is to provide necessary and sufficient conditions for interpolability. Our conditions below provide this, up to a small change in the function $L(\cdot)$ by an absolute constant. Let the ``self-smoothness constant'' $c_L$ of $\theta_L$ (previously defined in \eqref{def:theta}) be the smallest positive number for which $\theta_L$ is $c_LL(\cdot/c_L)$-inexactly smooth.
\begin{align}\label{primal-def-c_L-1D}
    c_L &= \inf \left\{c > 0: \theta_L \text{ is } cL(\cdot/c)\text{-inexactly smooth}\right\}
\end{align}
We can write this constant as $$c_L=\inf \left\{c > 0 : \forall r,s, g_s \in \partial \theta_L(s),\, \theta_L(r) \leq \theta_L(s)+g_s(r-s)+c\theta_L(r-s) \right\}$$
since $c \theta_L(s) = c \inf_{\delta \geq 0} \left\{\frac{L(\delta)}{2}s^2 + \delta\right\} = \inf_{\delta \geq 0} \left\{\frac{cL(\delta/c)}{2}s^2+\delta\right\}$.
The following theorem notes that this self-smoothness constant is well-defined and at most two.
\begin{theorem}\label{thm:finite-cL}
    If $L(\cdot)$ satisfies our assumptions, then $1 \leq c_L \leq 2$.
\end{theorem}

Necessary conditions for interpolation follow from any necessary condition for inexactly smooth convex functions, specialized to the points $x_i$. For example, the inequality~\eqref{def: inexactly smooth} with $y=x_i$ and $x=x_j$ and the subgradient inequality $f_i \geq f_j + \langle g_j,x_i-x_j\rangle$ are necessary. In the setting of $L$-smooth convex functions, \cite[Corollary 1]{taylor2017smooth} showed that the cocoercive inequalities $f_i \geq f_j + \langle g_j,x_i-x_j\rangle + \frac{1}{2L}\|g_i-g_j\|^2$ are necessary and sufficient.
The following theorem shows that the generalization of this cocoercive-type condition~\cref{lemma:L(delta) conjugate equivalence} remains necessary and sufficient for inexactly smooth interpolation, up to the self-smoothness constant $c_L$. 
 
\begin{theorem} \label{thm:interpolation-general}
   Let $L(\cdot)$ satisfy our assumptions. If a set of observations $\mathcal{H} =\{(x_i,f_i,g_i)\}_{i\in\Ical}$ is $L(\cdot)$-interpolable, then for each $i,j\in\Ical$,
\begin{equation}\label{eq:interpolation condition}
        f_i - f_j - \langle g_j, x_i-x_j\rangle - \frac{1}{2L(\delta)}\|g_i-g_j\|^2 + \delta \geq 0, \quad \forall \delta \geq 0 . 
    \end{equation} 
    Conversely, if the above condition holds then  $\mathcal{H}$ is ${c_L L(\cdot/c_L)}$-interpolable.
\end{theorem}

Recall our motivating example of $(\beta,p)$-\Holder smooth functions always satisfy an inexactly smooth bound with $L(\delta) = \left(\frac{1-p}{1+p}\frac{1}{2\delta}\right)^\frac{1-p}{1+p}\beta^\frac{2}{1+p}$ by \cref{prop:Equiv Char of HS functions}. Hence, from \eqref{primal-def-c_L-1D}, $c_L$ is exactly the maximum of $\left|1-s\right|^{\left(p+1\right)}-s^{\left(p+1\right)}+\left(p+1\right)s^{p}$ for $s \geq 0$, which is uniformly upper bounded by $c_L \leq 2^{1-p} \leq 2$
where the first inequality above is a direct application of the scalar ratio bound \eqref{eq:c_L bound} derived in the proof of \cref{thm:finite-cL}.\footnote{For $(\beta, p)$-\Holder smooth functions, $\theta_L(s) = \frac{\beta}{p+1}|s|^{p+1}$ and the supremum of $2^{1-p}$ in \eqref{eq:c_L bound} is attained when $r = -s$.}

This is tight when $p=1$ but may be slack for $p<1$. Note that the previous example \eqref{ex:tightness-Holder-condition} established that any interpolation theorem based on inexact quadratic bounds will be slack by at least a factor of $\left(\frac{p+1}{2p}\right)^p$. Combined, these give the following interpolation conditions for \Holder smooth functions, where we take the infimum over $\delta \geq 0$ in \eqref{eq:interpolation condition}.  Note that the sufficiency side can be tightened with the exact value of $c_L$ instead of the bound $2^{1-p}$ given above.

\begin{corollary}\label{cor:Holder_smooth_interp}
    If a set of observations $\mathcal{H} =\{(x_i,f_i,g_i)\}_{i\in\Ical}$ is interpolable by some convex, $(\beta,p)$-\Holder smooth function for $p \in (0,1]$, then for each $i,j\in\Ical$,
    \begin{equation}\label{eq:interpolation condition Holder}
        f_i - f_j - \langle g_j, x_i-x_j\rangle - \beta^{-1/p}\frac{p}{p+1}\|g_i-g_j\|^{(p+1)/p} \geq 0 .
    \end{equation} 
    Conversely, if the above holds then  $\mathcal{H}$ is interpolable by a convex, $\left(2\left(\frac{p+1}{4p}\right)^{p}\beta,p\right)$-\Holder smooth function.
\end{corollary}

Considering the limiting case as $p \to 0$, the condition above becomes
$$f_i-f_j-\inner{g_j}{x_i-x_j}-\iota_{[0,\beta]}(\|g_i-g_j\|) \geq 0$$ where $\iota_{[0,\beta]}$ is the indicator for the interval $[0,\beta]$. Hence the condition requires $f_i \geq f_j + \inner{g_j}{x_i-x_j}$ and $\|g_i-g_j\| \leq \beta$. However, in this $p=0$ special case, despite having $c_L = 2$, these are necessary and sufficient for $(\beta,0)$-interpolation.
\begin{theorem}\label{thm:interpolation-p0}
    Consider any set of observations $\mathcal{H}=\{(x_i, f_i, g_i)\}_{i \in \Ical}$. Then $\mathcal{H}$ is interpolable by a $(\beta,0)$-\Holder smooth, convex function if and only if $\mathcal{H}$ is $L(\delta) = \frac{\beta^2}{2\delta}$-interpolable if and only if for all $i, j \in \Ical$, $f_i \geq f_j + \inner{g_j}{x_i-x_j}$ and $\|g_i-g_j\| \leq \beta$.
\end{theorem}
Since $p=0$ admits a sharper exact interpolation theorem, further improvements for $p \in (0,1)$ may be possible. However, the previous example in~\eqref{ex:tightness-Holder-condition} establishes that a gap of at least $\left(\frac{p+1}{2p}\right)^p >1$ is necessary for any approach for \Holder smoothness with $p \in (0,1)$ based on inexact smoothness. As a result, Corollary~\ref{cor:Holder_smooth_interp}'s condition is only necessary and sufficient at $p=1$ and $p=0$.

\subsection{Proof of Interpolation Results} \label{subsec:ProofOfInterpolation}

In this section, we prove the above interpolation theory, utilizing the calculus results built in Section~\ref{subsec:calculus}. The primary work in proving \cref{thm:interpolation-general} and \cref{thm:interpolation-p0} is a construction of the necessary interpolating function given the claimed interpolation conditions.

\begin{proof}[{\bf Proof of Theorem~\ref{thm:finite-cL}}] 
The lower bound $c_L \geq 1$ is immediate. Set $s = 0$, and recall that $\theta_L(0) = 0$ with $0 \in \partial \theta_L(0)$ by \cref{lemma:theta_L-and-conjugate_properties}. Then $c \geq 1$ since $\eqref{primal-def-c_L-1D}$ requires $$\theta_L(r) \leq c \theta_L(r), \quad \forall r.$$
The rest of the proof proceeds in three steps. We first consider the properties
\begin{enumerate}
        \item $\|g_x-g_y\| \leq \|s_{xy}\|, \quad \forall g_x \in \partial f(x), g_y \in \partial f(y), \ \text{and} \ s_{xy}  \in \partial (\theta_L \circ \|\cdot\|)(x-y)$,
       \item $f(y) \leq f(x) + \inner{g_x}{y-x}+\theta_L(\|y-x\|), \quad \forall x,y \text{ with } g_x \in \partial f(x)$,
       \end{enumerate}
and demonstrate that the first implies the second. Second, we specialize to $f = \theta_L$ and bound the self-smoothness constant in \eqref{primal-def-c_L-1D} with the scalar ratio below
\begin{equation}\label{eq:c_L bound}c_L \leq \sup_{r \ne s} \left\{\frac{|g_r-g_s|}{|g_{rs}|} : g_r \in \partial \theta_L(r), g_s \in \partial \theta_L(s), g_{rs} \in \partial \theta_L(r-s)\right\}.\end{equation}
Third, we complete the proof by establishing a universal bound of $2$ on this quantity.

We first show ($1. \Longrightarrow 2.)$. Fix $x,y$ and arbitrary $g_x \in \partial f(x)$. Letting $g_t \in \partial f(x+t(y-x))$ and $s_{t} \in \partial (\theta_L \circ \|\cdot\|)(t(y-x))$ be arbitrary for any $ t\in [0,1]$, note
\begin{align}\label{application_nonsmooth_FTC}
\begin{split}
    f(y)-f(x)-\inner{g_x}{y-x}&=\int_0^1 \inner{g_t-g_x}{y-x}dt\\
    &\leq \int_0^1 \|g_t-g_x\|\|y-x\|dt\\
    &\leq \int_0^1 \|s_{t}\|\|y-x\|dt = \theta_L(\|y-x\|),
    \end{split}
\end{align}
where the nonsmooth fundamental theorem of calculus~\cite[Theorem 2.3.4]{fundamental_convex} implies the first equality, Cauchy-Schwarz implies the first inequality, and the hypothesis implies the second inequality. The last equality follows from two nonsmooth calculus rules. First, the nonsmooth chain rule~\cite[Corollary 16.72]{bauschke} applied to $(\theta_L \circ \|\cdot\|)$ gives $\|s_{t}\|\|y-x\| = \langle s_{t}, y-x\rangle$. Second, we consider the fundamental theorem of calculus again applied to $\theta_L \circ \|\cdot\|$ along the segment from $0$ to $y-x$. Rearranging yields $f(y) \leq f(x) + \inner{g_x}{y-x}+\theta_L(\|y-x\|)$.

We now construct the bound as in \eqref{eq:c_L bound}. By the definition of $c_L$ in \eqref{primal-def-c_L-1D}, for all $r, s \in \R$ with $g_s \in \partial \theta_L(s)$ the following bound holds $$\theta_L(r) \leq \theta_L(s)+g_s\cdot(r-s)+c_L \theta_L(r-s).$$

Suppose the following bound holds for some fixed $c > 0$, $$|g_r-g_s| \leq c|g_{rs}|, \quad \forall r, s \in \R, \  g_r \in \partial \theta_L(r), \ g_s \in \partial \theta_L(s), \ g_{rs} \in \partial \theta_L(r-s). $$ By applying $(1. \Longrightarrow 2.)$ as above, it holds that $$\theta_L(r) \leq  \theta_L(s) + g_s\cdot(r-s) + c\theta_L(r-s).$$ For $r = s$, the bound follows as $\theta_L(0) = 0$. So we can bound $c_L$ overall as in \eqref{eq:c_L bound}.

Finally, we bound this value by $2$. Without loss of generality, suppose $r > s$. Recall that $\theta_L$ is convex, even, and minimized at $0$. Respectively, the subdifferential $\partial\theta_L$ is monotone,  $\partial\theta_L(s) = -\partial\theta_L(-s)$, and $0 \in \partial\theta_L(0)$.

First consider $r \geq 0 \geq s$. Here $|r-s| = r + |s|$, so $g_{rs} \in \partial\theta_L(r+|s|)$. Monotonicity of the subdifferential with evenness of $\theta_L$ gives $|g_r| \leq |g_{rs}|$ and $|g_s| \leq |g_{rs}|$. Since $\theta_L$ minimizes and vanishes at $0$ (see \cref{lemma:theta_L-and-conjugate_properties}), $g_{rs} \geq 0$ as well. Hence,
$$|g_r - g_s| \leq |g_r| + |g_s| \leq 2 g_{rs}.$$

Now consider $r > s > 0$. Recall $\phi_L(s,\delta) = \frac{L(\delta)}{2}s^2 + \delta$ and denote $\delta_\star(s) = \argmin_{\delta \geq 0} \phi_L(s,\delta)$. Note that by \cref{lemma:phi_calc}, both $\partial \theta_L(s)$ and $\partial \theta_L(r)$ are singletons:
$$\partial \theta_L(s) = \{sL(\delta_\star(s))\}, \quad \partial \theta_L(r) = \{rL(\delta_\star(r))\}.$$
Adding the two inequalities $\phi_L(s,\delta_\star(s)) \leq \phi_L(s,\delta_\star(r))$ and $\phi_L(r,\delta_\star(r)) \leq \phi_L(r,\delta_\star(s))$ and simplifying terms that cancel, 
$$(r^2-s^2)(L(\delta_\star(r))-L(\delta_\star(s))) \leq 0 \Longrightarrow L(\delta_\star(r)) \leq L(\delta_\star(s)).$$
Therefore
\begin{align*}
    0 \leq g_r - g_s &= (r-s)L(\delta_\star(r))+s(L(\delta_\star(r))-L(\delta_\star(s)))\\
    &\leq (r-s)L(\delta_\star(r-s))=g_{rs},
\end{align*}
using monotonicity of the subgradient in the first inequality and monotonicity of $L(\delta_\star(\cdot))$ in the second.
The remaining case $0 > r > s$ follows by applying the above to $(-s,-r)$ and invoking evenness. With all cases exhausted, we may bound \eqref{eq:c_L bound} by $2$.
\end{proof}

\begin{proof}[{\bf Proof of \cref{thm:interpolation-general}}]
Suppose $f(x)$ is an $L(\cdot)$-inexactly smooth function that interpolates the observations $\mathcal{H}$. By the cocoercive-type condition 1.~from \cref{lemma:L(delta) conjugate equivalence} with $x = x_j$ and $y = x_i$, the claimed conditions hold for each $i, j \in \Ical$. 

For the reverse direction, define the function $r = (\max_{i\in \Ical} r_i^*)^*$ where
\begin{align*}
     r_i(x) &= f_i + \inner{g_i}{x-x_i}+\theta_L(\|x-x_i\|).
 \end{align*} 
 Calculating the conjugates of each $r_i$, one has $r_i^*(g) = -f_i + \inner{x_i}{g}+\theta_L^*(\|g-g_i\|) $.
 
 This claimed interpolation $r$ is convex as all conjugates are convex. Our calculus rules allow us to verify its inexact smoothness. Each individual $r_i$ is $c_L L(\cdot/c_L)$-inexactly smooth by the sum rule in Lemma~\ref{lemma:L(delta) sums and max calculus} since the linear term $f_i + \inner{g_i}{x-x_i}$ does not affect the smoothness and $\theta_L(\|x-x_i\|)$ is $c_L L(\cdot/c_L)$-inexactly smooth by Lemma~\ref{composing with abs to norm}. Then the $c_L L(\cdot/c_L)$-inexact smoothness of $r$ follows from the conjugate maximum formula in Lemma~\ref{lemma:L(delta) sums and max calculus}.

 All that remains is to verify that $r$ interpolates the given first-order information. Note that $r(x_j) = f_j$ and $g_j\in\partial r(x_j)$ if and only if $\max_{i\in \Ical} r_i^*(g_j) = -f_j+\langle x_j,g_j\rangle$ and $x_j \in \partial (\max_{i\in \Ical} r_i^*)(g_j)$. The key step in establishing this is showing that the interpolation conditions guarantee that $r_j^*$ attains this maximum at $g_j$. To see this, observe that for any $i\in\Ical$,
 \begin{align*}
     r_j^*(g_j) &= -f_j + \inner{x_j}{g_j}
     \geq -f_i + \inner{x_i}{g_j}+\theta_L^*(\|g_i-g_j\|)
     =r_i^*(g_j),
 \end{align*}
 where the inequality is our interpolation condition between $i$ and $j$ (taking the supremum over $\delta \geq 0$).
 From this, the fact that $r$ interpolates follows since $r^*(g_j)=-f_j+\langle x_j,g_j\rangle$ and $x_j \in \partial r_j^*(g_j)\subseteq \partial (\max_{i\in \Ical} r_i^*)(g_j)$.
\end{proof}

\begin{proof}[{\bf Proof of Theorem~\ref{thm:interpolation-p0}}]

Recall by \cref{prop:Equiv Char of HS functions} that a function $f$ is $(\beta, 0)$-\Holder smooth if and only if it is $L(\delta) = \frac{\beta^2}{2\delta}$-inexactly smooth. Therefore, the $(\beta, 0)$-\Holder interpolation of $\mathcal{H}$ is equivalent to $L(\delta) = \frac{\beta^2}{2\delta}$ interpolation.

Supposing there exists an $L(\delta) = \frac{\beta^2}{2\delta}$-inexactly smooth interpolating function, by \cref{thm:interpolation-general}, it must hold for all $i, j \in \Ical$ and $\delta \geq 0$ that $f_i -f_j - \inner{g_j}{x_i-x_j}-\frac{\delta}{\beta^2}\|g_i-g_j\|^2 +\delta \geq 0.$
Taking the infimum over $\delta \geq 0$, observe that the interpolation condition above is then equivalent to $f_i \geq f_j + \inner{g_j}{x_i-x_j}$ and $\|g_i-g_j\| \leq \beta.$

Conversely, consider the convex function $h(x) = \max \{f_i + \inner{g_i}{x-x_i}\}$. Since $\partial h \subseteq \conv{g_i}$, for any $g, g' \in \partial h$ with $g = \sum \alpha_i g_i$ and $g' = \sum \gamma_j g_j$, $(\beta,0)$-\Holder smooth follows as 
$\|g-g'\| \leq \beta$ by convexity of the norm. Finally, since $f_i \geq f_j + \inner{g_j}{x_i-x_j}$, the $i$th component of $h$ is active at $x_i$, and therefore $h(x_i) = f_i$ and $g_i \in \partial h(x_i)$. Hence, $h$ interpolates the observations. 
\end{proof}

\subsection{Inexactly Smooth Convex PEPs} \label{subsec:PEPs}
Our interpolation theorems enable PEPs~\cite{drori2014performance,taylor2017smooth,taylor2017composite} for inexactly smooth problems. For a numerical PEP study in the case of gradient descent on convex $(\beta,p)$-\Holder smooth functions, see~\cite[Section 4.4]{kamri2025}. Our theorems provide a structured way to analyze $N$-step fixed-step first-order methods (FSFOM). These methods, parameterized by a lower triangular $W\in\mathbb{R}^{N \times N}$, iterate
\begin{equation*}
    x_n = x_0 - \sum_{i=0}^{n-1} W_{n, i} g_i
\end{equation*}
for $n=1,\dots,N$, where $g_i\in\partial f(x_i)$ are subgradients computed at each iteration.


Consider minimizing an $L(\cdot)$-inexactly smooth, convex function $f$ with a minimizer $x_\star$ satisfying $\|x_0-x_\star\|\leq D$. The Performance Estimation Problem for some $W$ finds the worst final objective $f(x_N)-f(x_\star)$ over all possible problem instances. Formally,
\begin{equation}\label{eq:inexact-PEP}
     \pPEP{true} = \begin{cases}\displaystyle\max_{f, \, x_0 \in \R^d } & f(x_N)-f(x_\star)  \\
    \text{s.t.} & x_n = x_{0} - \sum_{i=0}^{n-1} W_{n,i} g_i \quad \forall n = 1, 2, \dots, N\\
    & \|x_0-x_\star\|^2 \leq D^2 \\
    & f \text{ is $L(\cdot)$-inexactly smooth, convex, and minimized at $x_\star$}. \\
    \end{cases}
  \end{equation}

The key observation to tractably solve PEPs is that the algorithm's trajectory and performance depend only on the first-order information at the iterates $x_0,\dots,x_N$ and at the minimizer $x_\star$. For ease in referring to such points, we consider the index set $ \Ical_N^\star = \{0, 1, \dots, N, \star\}$. The relevant quantities to algorithmic performance are then
\begin{equation*}
    x_i, \qquad f_i = f(x_i), \quad g_i\in\partial f(x_i) \qquad \forall i \in \Ical_N^\star
\end{equation*}
where we require $g_\star=0$ at the minimizer $x_\star$. From Theorem~\ref{thm:interpolation-general}, we know that the following must be nonnegative for all $i,j\in \Ical_N^\star$ and all $\delta_{i,j}\geq 0$
\begin{align*}
    \mathcal{D} &= D^2 - \|x_0-x_\star\|^2 \geq 0\\
    \mathcal{Q}_{i,j,\delta_{i,j}} &= f_i - f_j - \langle g_j, x_i-x_j\rangle - \frac{1}{2L(\delta_{i,j})}\|g_i-g_j\|^2 + \delta_{i,j} \geq 0.
\end{align*}
Note this can be reduced to a finite set of inequalities by taking a supremum over $\delta_{i,j}$, resulting in a compressed nonnegative condition $\mathcal{Q}_{i,j} =  f_i - f_j - \langle g_j, x_i-x_j\rangle - \theta_L^*(\|g_i-g_j\|)\geq 0$.
Using the sufficiency side of Theorem~\ref{thm:interpolation-general}, we know an $L(\delta)$ inexactly smooth instance exists agreeing with the observed first-order information whenever one has $\mathcal{D}\geq 0$ and for all $i,j\in \Ical_N^\star$ and all $\delta_{i,j}\geq 0$
\begin{align*}
    \mathcal{Q}^{c_L}_{i,j,\delta_{i,j}} &= f_i - f_j - \langle g_j, x_i-x_j\rangle - \frac{c_L}{2L(c_L\delta_{i,j})}\|g_i-g_j\|^2 + \delta_{i,j}\geq 0.
\end{align*}
Again, define a compressed notation $\mathcal{Q}^{c_L}_{i,j} =  f_i - f_j -\langle g_j, x_i-x_j\rangle - \theta_{L(c_L\cdot)/c_L}^*(\|g_i-g_j\|)$.

From these necessary and sufficient conditions, we define a closely related optimization problem over finitely many variables. An upper bound on $\pPEP{true}$ is 
\begin{equation}\label{eq:Inexact-Primal-infinite-interpolation}
    \pPEP{interp} = \begin{cases}\displaystyle\max_{x,f,g \in (\R^d \times \R \times \R^d)^{N+2}} & f_N-f_\star  \\
    \mathrm{s.t.} & x_n = x_{0} - \sum_{i=0}^{n-1} W_{n,i} g_i \quad \forall n = 1, 2, \dots, N\\
    & g_\star = 0\\
    & \mathcal{D} \geq 0 \\
    & \mathcal{Q}_{i,j} \geq 0 \qquad  \forall i, j \in \Ical_N^\star.
    \end{cases}
\end{equation}
Denote the lower bounding problem with stricter $\mathcal{Q}_{i,j}^{c_L}\geq 0$ constraints by $\pPEP{interp}^{c_L}$.

The standard approach to further simplifying such PEP reformulations is to consider a Gram change of variables. Define $P = [x_0-x_\star \ | \ g_0 \ |\  g_1 \ |\  \dots\ | \ g_N]$ and the Gram matrix of all inner products between these vectors as $G=P^TP$. After substituting the equality definitions of  $x_n = x_{0} - \sum_{i=0}^{n-1} W_{n,i} g_i$ and $g_\star=0$ into the remaining inequality constraints, observe that $\mathcal{D}$ and $\mathcal{Q}_{i,j}$ are linear in $F=[f_\star,f_0,\dots, f_N]^T$ and concave in $G$ (to see this, note $\theta_L^*\circ\sqrt{\cdot}$ is convex since it takes the form of a supremum of affine functions). We denote these concave functions by $\mathcal{D}(F,G)$ and $\mathcal{Q}_{i,j}(F,G)$. Changing over to these as the variables makes all of the above constraints convex. Additionally, as a Gram matrix, we require $G$ to be positive semidefinite.  We denote this by
\begin{equation}\label{eq:PEP_gram}
    \pPEP{interp} \leq \pPEP{gram} = \begin{cases}\displaystyle\max_{F,G} & f_N-f_\star  \\
    \mathrm{s.t.} & \mathcal{D}(F,G) \geq 0 \\
    & \mathcal{Q}_{i,j}(F,G) \geq 0 \qquad  \forall i, j \in \Ical_N^\star\\
    & G\succeq 0.
    \end{cases}
\end{equation}
Provided $d\geq N+2$, one can factor $G$ to recover $P$, making this reformulation exact.

The above PEP formulation~\eqref{eq:PEP_gram} is a convex problem in $F,G$. The Lagrange dual problem certifies upper bounds on $f_N-f_\star$ via nonnegative weights $\nu,\lambda_{i,j}$ for the inequality constraints and a positive semidefinite $Z$. We denote the dual by $\dPEP{gram}$. The following theorem relates all of the defined PEP problems, establishing, in particular, that strong duality holds above under a mild regularity condition on $W$.

\begin{theorem}\label{thm:strong duality}
    Let $L(\cdot)$ satisfy our assumptions and suppose $d \geq N+2$. Given $W_{i,i-1} \ne 0$ for all $i = 1, \dots,  N$, strong duality holds, $\pPEP{gram}=\dPEP{gram}$. Moreover, $$\pPEP{gram}^{c_L} = \dPEP{gram}^{c_L} = \pPEP{interp}^{c_L} \leq \pPEP{true} \leq \pPEP{interp} = \pPEP{gram}= \dPEP{gram} .$$
\end{theorem}

\subsection{Proof of \cref{thm:strong duality}}

First, we show that for any $L(\cdot)$ satisfying our assumptions, we can construct a Huber-type function that is convex and $L(\cdot)$-inexactly smooth in an arbitrary dimension $d$. From this, it follows that $\pPEP{gram}$~\eqref{eq:PEP_gram} has a Slater point, given $W$ has nonzero diagonal. We defer the proofs of these two results, \cref{lemma:inexactly_smooth_construction} and \cref{lemma:slater point}, to Appendix~\ref{sec: appendix PEP}.

\begin{lemma}\label{lemma:inexactly_smooth_construction}
    Let $L(\cdot)$ satisfy our assumptions and $R> 0$. Then the function $$h(x) = \begin{cases}
        \frac{k}{2}\|x\|^2 & \|x\| \leq R\\
        kR\|x\|-\frac{kR^2}{2} & \|x\| > R
    \end{cases}, \quad k := \inf_{\delta \geq 0} \left\{\frac{L\left(\delta\right)+\sqrt{L\left(\delta\right)^{2}+\frac{2\delta L\left(\delta\right)}{R^2}}}{2}\right\}  $$
    is $L(\cdot)$-inexactly smooth and $k > 0$.
    Further, for any $Q\succ 0$, with $\|Q\|_{\mathrm{op}} \leq 1$, $$f(x) := h(Q^{1/2}x) =  \begin{cases}
        \frac{k}{2}\|x\|_Q^2 & \|x\|_Q \leq R\\
        kR\|x\|_Q-\frac{kR^2}{2} & \|x\|_Q > R
    \end{cases}$$ is $L(\cdot)$-inexactly smooth as well, where $\|x\|_Q := \sqrt{x^TQx}$.
  
\end{lemma}

\begin{lemma}\label{lemma:slater point}
    Let $L(\cdot)$ satisfy our assumptions. Under the assumption that $W_{i,i-1} \ne 0$ for each $i = 1, \dots, N$, there exists feasible $F, G$ for the convex problem $\pPEP{gram}$~\eqref{eq:PEP_gram} such that $G \succ 0$ and $\mathcal{Q}_{i,j} (F,G) > 0$ for all $i \ne j \in \Ical_N^\star$. 
\end{lemma}

Existence of a Slater point from the above \cref{lemma:slater point} ensures strong duality between the convex primal and dual formulations $\pPEP{gram}$ and $\dPEP{gram}$ (see~\cite[Theorem 28.2]{rockafeller}). Similarly, the lower bounding formulations of $\pPEP{gram}^{c_L}$ and $\dPEP{gram}^{c_L}$ are equal. Then, recalling that if $d \geq N+2$, one may always factor positive semidefinite $G$ into its Cholesky form, the Gram reformulations are exact, $\pPEP{interp} = \pPEP{gram}$ and $\dPEP{gram}^{c_L} = \pPEP{interp}^{c_L}$. Finally, the inequalities $\pPEP{interp}^{c_L} \leq \pPEP{true} \leq \pPEP{interp}$ follow from our interpolation theorem (Theorem~\ref{thm:interpolation-general}).

 \section{A Constructive Approach to Optimized Algorithm Design}\label{sec:algo design}

The above interpolation and performance estimation theory directly enables optimized algorithm designs via the constructive approach of Drori and Taylor~\cite{constructive_approach}. Below, we introduce minimax optimality among algorithms and formalize their constructive PEP approach. Then we describe three resulting designs targeting different levels:

Section~\ref{subsec:BSD} presents an exactly minimax optimal method for $(\beta,0)$-\Holder smooth functions (i.e., $L(\delta) = \frac{\beta^2}{2\delta}$-inexactly smooth). Section~\ref{subsec:asymptotic_OGM} presents an optimized method for any $(\beta,p)$-\Holder smooth functions by optimizing performance over $L(\delta)=\kappa/\delta^q$-inexactly smooth functions. This method is big-O optimal and we conjecture that it possesses the optimal leading coefficient. Finally, Section~\ref{subsec:OBL} presents a universal, parameter-free method for any $L(\cdot)$.
The resulting method offers optimized big-O optimal guarantees for any \Holder smooth setting or sums thereof.

Optimality in algorithm design is defined in terms of a considered family of problem instances and a family of algorithms. 
As problems, given a function $L(\cdot)$ and some $D>0$, we consider minimizing $L(\cdot)$-inexactly smooth convex functions from an initialization $x_0$ with $\|x_0-x_\star\|\leq D$. Denote this class
\begin{align}\label{def:problem_family}
\begin{split}
\mathtt{P}_{L(\cdot), D} := \{(f, x_0) \colon &f \text{ is convex, $L(\cdot)$-inexactly smooth}, \\ &\text{attains a minimizer at $x_\star$ with $\|x_0-x_\star\| \leq D,$}\\
&\text{and a subgradient oracle $g$ such that $g(x) \in \partial f(x)$}\}.
\end{split}
\end{align}
Note that $\mathtt{P}_{L(\cdot), D}$ contains problems over every dimension $d$. As algorithms, we consider $N$-step methods constructing points $x_1,\dots,x_N$ satisfying the subgradient span condition $ x_n \in x_0 + \spann{g_0, \dots, g_{n-1}}$ with $g_i=g(x_i)$. Denote the set of such methods by $\Algclass{span}$. Note that this contains, for example, all FSFOM. 

The task of finding the algorithm with the best worst-case performance against a family of problem instances, measured by final objective gap, is then
\begin{equation}\label{def:minimax_opt}
    \min_{\mathtt{a} \in \Algclass{span}}\max_{(f,x_0) \in \mathtt{P}_{L(\cdot), D}} f(x_N) - f(x_\star).
\end{equation} 
We say an algorithm $\mathtt{a}$ is minimax optimal if it attains the above min. It is big-O optimal if, for all $N$, it remains within a constant factor of attaining this rate. Proving minimax optimality of methods historically has been done by finding a hard problem instance for all algorithms~\cite{nemirovskilowerbounds}. By weak duality, \eqref{def:minimax_opt} is lower bounded by 
\begin{equation}\label{def:maximin_hard}
  \max_{(f,x_0) \in \mathtt{P}_{L(\cdot), D}}  \min_{\mathtt{a} \in \Algclass{span}} f(x_N) - f(x_\star).
\end{equation} 
Strong duality often holds. For constant $L(\cdot)=L$, this is established by the OGM method~\cite{kim2016optimized} and matching lower bound~\cite{drori2017exact}. For $L(\delta)=\kappa/\delta$, our Theorem~\ref{thm:minimax_BSFOM} below establishes strong duality.

\subsection{Extension of the Constructive Approach of~\cite{constructive_approach}}\label{subsec:Constructive Approach}

Here we consider the constructive approach to algorithm design problems with semidefinite programming PEPs of Drori and Taylor~\cite{constructive_approach}. These techniques generalize directly to the inexactly smooth setting with its convex PEP formulation.
To this end, consider the following hypothetical first-order method, iterating
\begin{align*}
    x_n &\in \argmin \{f(x) \colon x \in x_0 + \mathrm{span}\{g_0, g_1, \dots, g_{n-1}\}\}\\
   g_n &\in \partial f(x_n) \ \text{such that } \inner{g_n}{g_i} = 0, \ \forall 0 \leq i < n.
\end{align*}
For the sake of this motivation, assume the above $\argmin$ is nonempty and consider any selection of (possibly adversarial) $x_n$ and $g_n$. Following the nomenclature of~\cite{constructive_approach}, we refer to such a hypothetical algorithm as a Greedy First-Order Method (GFOM).

The constructive approach proceeds by (i) solving the PEP problem associated with GFOM, (ii) computing dual multipliers proving its convergence rate, and (iii) identifying a fixed-step first-order method with the same performance (and, in fact, the same PEP proof) as GFOM. Formally, denote GFOM's PEP by
\begin{equation}\label{constructive_approach_PEP_true}\pAlg_{true} = \begin{cases} \max_{(f,x_0)} & f(x_N)-f_\star\\ 
   \mathrm{s.t.}  & x_n \text{ is constructed by some GFOM}\\
   & (f,x_0) \in  \mathtt{P}_{L(\cdot), D}.
    \end{cases}\end{equation}
From our interpolation theory, we derive the following upper bounding problem
\begin{equation}\label{constructive_approach_PEP}
\pAlg_{true} \leq \pAlg_{interp} = \begin{cases} \max_{x_i, f_i, g_i} & f_N-f_\star\\ 
   \mathrm{s.t.}&   \inner{g_i}{g_j} = 0, \quad \forall 0 \leq j < i = 1, \dots, N\\
& \inner{g_i}{x_j-x_0} = 0, \quad \forall 1\leq j\leq i = 1,\dots, N\\
& g_\star = 0\\
& \mathcal{D} \geq 0\\
 & \mathcal{Q}_{i,j}\geq 0, \quad i,j \in \mathcal{I}_N^\star.
    \end{cases}\end{equation}
Observe that after substituting the orthogonality constraints, this is a convex problem in the variables $f_i$, $\inner{g_i}{x_j}$, and $\|g_i\|^2$. The corresponding dual problem can be formulated as follows (the derivation of this program is deferred to \cref{sec: appendix algo design})
\begin{equation}\label{upper_triangular_dual_pep}\dAlg_{interp} = \begin{cases} \min_{\lambda, {t}, s}
   & {\frac{1}{2}D^2s + \sum_{i,j \in \mathcal{I}_N^\star} \lambda_{i,j }\Lleft(\lambda_{i,j}/ t_{i,j})}\\
    \mathrm{s.t.} & \lambda_{\star,j}^2/s \leq  \sum_{i=0}^{j-1}   t_{i,j}+ \sum_{i=j+1}^N  t_{j,i} +  t_{\star,j} -t_{j,\star} , \quad  \forall j\\
    &\sum_{i=j+1}^N \lambda_{j,i} - \sum_{i=0}^{j-1} \lambda_{i,j}  = \lambda_{\star,j}-\lambda_{j,\star}, \quad \forall j \ne N\\
    &
    \sum_{i=0}^{N-1}\lambda_{i,N} = \sum_{i=0}^{N-1}\lambda_{\star,i}-\lambda_{i,\star}\\
    & \sum_{i=0}^N \lambda_{\star,i} -\lambda_{i,\star} = 1\\
   &  \lambda \geq 0, \  t \geq 0, \ s \geq 0.
\end{cases}\end{equation}
Recall $s \Lleft(s)$ is convex by \cref{lemma:equivalences-of-convexity-assumptions}. From this, the convexity of the dual objective is clear as it uses the perspective function of this~\cite[Section 3.2.6]{boyd2004convex}.

In particular, solving this dual problem gives certificates $\lambda, t, s$ that prove a convergence rate for GFOM. The ``Subspace Search Elimination Procedure'' (SSEP) of~\cite{constructive_approach} shows how to construct an FSFOM for which this same certificate also proves a convergence rate. Namely, after a notational
rearrangement of~\cite[Corollary 1]{constructive_approach}, set $z_0=x_0$, $z_1 = x_0 - \lambda_{\star,0}g_0$, and iterate for $n=1,\dots,N$
\begin{align}\label{eq:dual_generated_algo}
    \begin{split}
    x_n &= \frac{\sum_{i=0}^{n-1} \left(\lambda_{i,n} x_i - t_{i,n} g_i\right) + \lambda_{\star,n} z_n}{\sum_{i=0}^{n-1} \lambda_{i,n} + \lambda_{\star,n}}\\
    z_{n+1} &= z_n-\lambda_{\star,n}g_n.
    \end{split}
\end{align} 
The following lemma explicitly connects this to~\cite{constructive_approach} and extracts a guarantee.

\begin{lemma}\label{lem:equivalence}
Fix $N \geq 1$, $D > 0$ and let $(\lambda, t, s)$ be feasible
for~\eqref{upper_triangular_dual_pep}. Suppose $\Lambda_n := \sum_{i=0}^{n-1} \lambda_{i,n} + \lambda_{\star,n} \ne 0$ for all $n = 1,\dots,N$.
Then the method~\eqref{eq:dual_generated_algo} is exactly the fixed-step method
of~\cite[Corollary~1]{constructive_approach} under the identifications
\begin{align}\label{eq:dual-dictionary}
\begin{split}
    \tilde\gamma_{n,n} = \Lambda_n, \quad
    \tilde\gamma_{n,j} = -\lambda_{j,n} \qquad &\forall j = 1, \dots, n-1,\\
    \tilde\beta_{n,j}  = t_{j,n} + \lambda_{\star,n}\lambda_{\star,j} \qquad &\forall j = 0, \dots, n-1.
\end{split}
\end{align}
Further, for any problem instance $(f, x_0) \in \mathtt{P}_{L(\cdot),D}$, this method inherits the dual objective bound of 
$f(x_N) - f_\star \leq \frac{1}{2}D^2s + \sum_{i,j \in \mathcal{I}_N^\star} \lambda_{i,j}\Lleft(\lambda_{i,j}/t_{i,j}).$
\end{lemma}

\begin{proof}
Recall that we set $z_0 = x_0$, $z_1 = x_0 - \lambda_{\star,0}g_0$ and, for $n = 1,\dots,N$,
\begin{equation}\label{eq:equiv-iteration}
    x_n = \frac{\sum_{i=0}^{n-1}\lambda_{i,n}x_i
        + \lambda_{\star,n}z_n
        - \sum_{i=0}^{n-1} t_{i,n}g_i}{\Lambda_n},
    \qquad z_{n+1} = z_n - \lambda_{\star,n}g_n.
\end{equation}
Unrolling the $z$-recurrence gives
$z_n = x_0 - \sum_{k=0}^{n-1}\lambda_{\star,k}g_k$. Substituting
into~\eqref{eq:equiv-iteration},
$$
    \Lambda_nx_n
    = \sum_{i=0}^{n-1}\lambda_{i,n}x_i
        + \lambda_{\star,n}x_0
        - \sum_{i=0}^{n-1}\left(t_{i,n} + \lambda_{\star,n}\lambda_{\star,i}\right)g_i.
$$
Subtracting $\Lambda_nx_0 = \sum_{i=0}^{n-1}\lambda_{i,n}x_0 + \lambda_{\star,n}x_0$, dividing by $\Lambda_n\neq 0$,
and rearranging gives
\begin{equation}\label{eq:equiv-DT-form}
    x_n = x_0
        + \sum_{i=1}^{n-1}\frac{\lambda_{i,n}}{\Lambda_n}(x_i - x_0)
        - \sum_{i=0}^{n-1}\frac{t_{i,n} + \lambda_{\star,n}\lambda_{\star,i}}{\Lambda_n}g_i.
\end{equation}
The correspondence given in~\eqref{eq:dual-dictionary} applied to~\eqref{eq:equiv-DT-form} is exactly the fixed-step
formula given in~\cite[Corollary 1]{constructive_approach}. From this, since
$\tilde\gamma_{n,n} = \Lambda_n \ne 0$ by assumption, their corollary implies $f(x_N) - f_\star \le \frac{1}{2}D^2s + \sum_{i,j \in \mathcal{I}_N^\star} \lambda_{i,j}\Lleft(\lambda_{i,j}/t_{i,j})$.
\end{proof}

\subsection{An Exactly Optimal Method for $(\beta,0)$-\Holder Smooth Convex Minimization}\label{subsec:BSD}

As a first application, consider $(\beta,0)$-\Holder smooth convex minimization.
By \cref{prop:Equiv Char of HS functions}, this is exactly the class of $L(\delta) = \frac{\beta^2}{2\delta}$-inexactly smooth convex functions. Our interpolation theorem is tight for this setting, by \cref{thm:interpolation-p0}. Hence, from \cref{thm:strong duality}, we have $\pPEP{true} = \pPEP{interp}$ and similarly, $\pAlg_{true}=\pAlg_{interp}$. 

This class is closely related to the well-studied model of $M$-Lipschitz convex minimization. While any $M$-Lipschitz convex function is necessarily $(2M, 0)$-\Holder smooth, $(\beta,0)$-\Holder smoothness combined with existence of a minimizer implies $\beta$-Lipschitzness. As a result, these distinct models are equivalent only up to differences in universal constants. Which of these two models is more relevant is a modeling choice; as one nice property, $(\beta,0)$-\Holder smoothness is tilt invariant (i.e., preserved under addition with linear functions). 

The constructive approach in~\cite{constructive_approach} gave a minimax optimal method for Lipschitz convex problems. Applied to $(\beta,0)$-\Holder smooth convex problems below yields a similar (but distinct) minimax optimal method here. 

\subsubsection{A Minimax Optimal Algorithm}
With $L(\delta) = \frac{\beta^2}{2\delta}$, \eqref{upper_triangular_dual_pep} simplifies to
$$\dAlg_{interp} = \begin{cases} \min_{\lambda, t \geq 0} & \frac{D^2+\sum_{0\leq i<j\leq N}\beta^2 t_{i,j}}{2\left(\sum_{i=0}^N \lambda_{\star,i}-\lambda_{i,\star}\right)}\\
    \mathrm{s.t.}   & \sum_{i=0}^{j-1}  - t_{i,j}+ \sum_{i=j+1}^N -t_{j,i} - t_{\star,j} +t_{j,\star} + \lambda_{\star,j}^2 \leq 0, \quad  \forall j\\
    &  \sum_{i=j+1}^N \lambda_{j,i} -\sum_{i=0}^{j-1} \lambda_{i,j}  = \lambda_{\star,j} -\lambda_{j,\star}, \quad \forall j \ne N\\
    &
    \sum_{i=0}^{N-1}\lambda_{i,N} = \sum_{i=0}^{N-1}\lambda_{\star,i}-\lambda_{i,\star}.\end{cases}$$
A simple calculation verifies that the following is a feasible solution, setting
\begin{align}\label{eq:feasible_SSEP_dual_sol}
\begin{split}
    \lambda_{\star,i} = \frac{D\sqrt{2}}{\beta\sqrt{N+1}}, \quad &i = 0, \dots, N, \quad
    \lambda_{i,i+1} = \frac{\sqrt{2}D(i+1)}{\beta \sqrt{N+1}}, \quad i = 0, \dots, N-1,\\
    t_{i,j} &= \frac{2D^2}{\beta^2 N(N+1)}, \quad i < j = 1, \dots, N
    \end{split}
\end{align}
and all other variables as zero. The dual objective value of this candidate solution is
$$\frac{D^2+\sum_{i,j}\beta^2 t_{i,j}}{2\sum_{i=0}^N \lambda_{\star,i}} = \frac{\beta D}{\sqrt{2(N+1)}}.$$
\Cref{alg:SSEP_BSD} presents the algorithm induced by this certificate, reformulating~\eqref{eq:dual_generated_algo}.

\begin{algorithm}
\caption{SSEP Method for $(\beta,0)$-\Holder Smooth Convex Minimization}
\begin{algorithmic}\label{alg:SSEP_BSD}
\Input $x_0 \in \mathbb{R}^d$, iteration budget $N$, parameters $\beta, \, D$
\For{$n = 1, \ldots, N$}
    \State $y_n = \frac{n}{n+1}x_{n-1} + \frac{1}{n+1}x_0$
    \State $d_n = \frac{1}{n+1}\sum_{j=0}^{n-1} g_j$
    \State $x_n = y_n - \frac{\sqrt{2}D\sqrt{N+1}}{\beta N} d_n$
\EndFor
\end{algorithmic}
\end{algorithm}

Trivial modifications of the known hard instance for Lipschitz convex problems establish a matching lower bound via~\eqref{def:maximin_hard}. Hence, this method is minimax optimal. 
\begin{theorem}\label{thm:minimax_BSFOM}
    For any $N\geq 1, D>0$ and any convex, $(\beta, 0)$-\Holder smooth $f$ with minimizer $x_\star$ satisfying $\|x_0-x_\star\| \leq D$,  \Cref{alg:SSEP_BSD} solves~\eqref{def:minimax_opt}, having its terminal iterate be guaranteed to have
    $$f(x_N) - f(x_\star) \leq \frac{\beta D}{\sqrt{2(N+1)}}.$$
\end{theorem}
\begin{proof}
The claimed guarantee is immediate from \cref{lem:equivalence}. The lower bound follows from a standard maximin hard instance design for nonsmooth optimization: Let $d = N+1$ and $e_i$ denote the $i$th standard basis vector. Consider $x_0=0$,
$$ f(x) = \frac{\beta}{\sqrt{2}}\max\left\{\max_{i=1, \dots, N+1} \inner{x}{e_i}, -\frac{D}{\sqrt{N+1}}\right\} $$
and the subgradient oracle choosing $$
g(x) = \begin{cases}
\frac{\beta}{\sqrt{2}}e_{i_\star} & f(x) > -\frac{\beta D}{\sqrt{2(N+1)}}\\
0 & f(x) = -\frac{\beta D}{\sqrt{2(N+1)}}
\end{cases}, \quad i_\star = \min\{i : f(x) = \beta\inner{x}{e_i}/\sqrt{2}\}.$$
It is easy to verify that $f$ is $(\beta,0)$-\Holder smooth, convex, and has $\|x_0-x_\star\|=D$ with minimizer at $x_\star = -\frac{D}{\sqrt{N+1}}\sum_{i=1}^{N+1} e_i$. Then the minimal objective value is $f(x_\star) = -\frac{\beta D}{\sqrt{2(N+1)}}$. A standard zero-chain argument (see~\cite[Theorem A.1]{drori2016optimal}) establishes that any algorithm in $\Algclass{span}$ has $f(x_N) \geq 0$. Hence, every subgradient span method has $ f(x_N) - f(x_\star) \geq \frac{\beta D}{\sqrt{2(N+1)}}$. By matching our upper bound, this proves optimality.
\end{proof}




\subsubsection{Comparison with Lipschitz Problem Class}
Given $(\beta,0)$-\Holder smoothness differs from Lipschitz continuity only up to small constants, here we briefly investigate the performance of algorithms across these classes. Namely, we consider three known optimal methods for $M$-Lipschitz convex minimization: (i) The subgradient method with constant stepsize $h = D/(M\sqrt{N+1})$ and final iterate averaging~\cite[Section 3.1]{bubeck2015convex}, (ii) The subgradient method with nonconstant stepsizes but optimal final iterate~\cite{zamani2025exact}, (iii) The SSEP induced method of~\cite{constructive_approach}. Each of these methods has an equal worst-case performance of $MD/\sqrt{N+1}$.

Surprisingly, when numerically solving our PEP~\eqref{eq:PEP_gram} for $(\beta,0)$-\Holder smooth convex problems, these three methods appear to possess identical convergence rates for our setting of interest as well. See \cref{fig:PEP_measures_both}. This holds whether we set $M=\beta$ or heuristically set $M=\beta/\sqrt{2}$. In either case, these methods are strictly suboptimal, performing worse than \Cref{alg:SSEP_BSD}. With $M=\beta/\sqrt{2}$, the optimal methods for the Lipschitz setting appear to have the optimal asymptotic coefficient for the $(\beta,0)$-\Holder smooth setting, making them only suboptimal in little-o terms. 

\pgfplotstableread{
N    bounded_subgrad
2    0.5278214039126697
4    0.3689511597399028
6    0.3014716301722239
8    0.2615740359995039
10   0.2343564502751046
12   0.2142350576675436
14   0.1985671858943041
16   0.1859140302821852
18   0.1754157410847461
20   0.1665208169248241
22   0.1588579687187615
24   0.1521662086403103
26   0.1462560603836161
28   0.1409861989530882
30   0.1362485762975360
32   0.1319588112750648
34   0.1280511330249295
36   0.1244710024986351
38   0.1211761870250556
40   0.1181295335356507
42   0.1153018200680053
44   0.1126687620075987
46   0.1102076748119571
48   0.1079012331257297
50   0.1057341335337608
}\subgradavgA

\pgfplotstableread{
N   minimax_rate
2   0.4082482904638631
4   0.3162277660168379
6   0.2672612419124244
8   0.2357022603955159
10  0.2132007163556104
12  0.1961161351381840
14  0.1825741858350554
16  0.1714985851425088
18  0.1622214211307625
20  0.1543033499620919
22  0.1474419561548971
24  0.1414213562373095
26  0.1360827634879543
28  0.1313064328597225
30  0.1270001270001905
32  0.1230914909793327
34  0.1195228609334394
36  0.1162476387438193
38  0.1132277034144596
40  0.1104315260748465
42  0.1078327732034384
44  0.1054092553389460
46  0.1031421246258793
48  0.1010152544552211
50  0.0990147542976674
}\minimaxdataA
\pgfplotstableread{
N     bounded_subgrad
2     0.4422673285654834
4     0.3241327326465001
6     0.2704423900086844
8     0.2373386988116757
10    0.2141695846657106
12    0.1967444561056841
14    0.1830087041595155
16    0.1718136132674596
18    0.1624581623099317
20    0.1544867640178836
22    0.1475874830863200
24    0.1415388716352179
26    0.1361792572602712
28    0.1313866460203770
30    0.1270680542392557
32    0.1231493639614786
34    0.1195722528827500
36    0.1162903118061764
38    0.1132652104267546
40    0.1104643087129890
42    0.1078618338610400
44    0.1054352350655755
46    0.1031650185088721
48    0.1010358270926593
50    0.0990333893726637
}\subgradavgB

\begin{figure}[htbp]
\newlength{\PEPgraphheight}
\setlength{\PEPgraphheight}{5cm}
    \centering

        \begin{tikzpicture}
        \begin{axis}[
            axis lines=left,
            width=0.8\textwidth,
            height=\PEPgraphheight,
            xmode = log, ymode = log,
            xlabel={$N$},
            ylabel={$f_N-f_\star$},
           every axis y label/.append style={yshift=10pt},
            xmin=2, xmax=50,
            ymin=0.1, ymax=0.65,
            grid=major,
            grid style={gray!30},
            legend style={
                at={(0.97,0.97)},
                anchor=north east,
                font=\scriptsize,
                draw=black!50,
                fill opacity = 0.9,
                text opacity = 1,
                cells={anchor=west}
            },
            tick label style={font=\small},
            label style={font=\normalsize},
            unbounded coords=jump,
        ]
 
        \addplot[color=cyan!70!white, solid, line width=2pt, mark=none]
            table [x=N, y=bounded_subgrad] {\subgradavgA};
        \addlegendentry{SGM $M = \beta$}
 
        \addplot[color=blue!90!black, solid, line width=2pt, mark=none]
            table [x=N, y=bounded_subgrad] {\subgradavgB};
        \addlegendentry{SGM $M = \beta/\sqrt{2}$}

        \addplot[color=black, dashed, line width=2pt, mark=none]
            table [x=N, y=minimax_rate] {\minimaxdataA};
        \addlegendentry{Minimax optimal rate}
 
        \end{axis}
        \end{tikzpicture}

         \caption{Suboptimality of the subgradient method~\cite[Section 3.1]{bubeck2015convex} on $(\beta,0)$-\Holder smooth functions. Other optimal subgradient methods are omitted as they had numerically identical PEP values.}

    \label{fig:PEP_measures_both}
\end{figure}
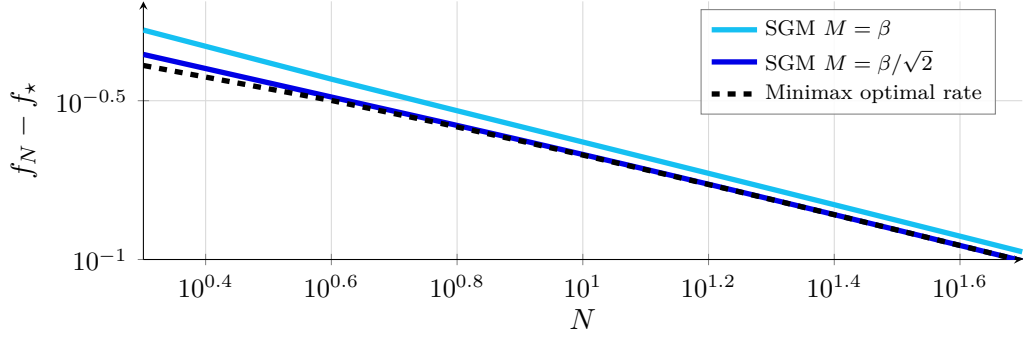

\subsection{Asymptotically Optimized Methods for \Holder Smooth Minimization}\label{subsec:asymptotic_OGM}
Next, we consider $(\beta,p)$-\Holder smooth functions with $p\in[0,1]$. The limiting cases now have known optimal methods (the above method when $p=0$ and OGM~\cite{kim2016optimized} when $p=1$). We recall by \cref{prop:Equiv Char of HS functions} that if a function is $(\beta,p)$-\Holder smooth then it is $L(\delta) = \kappa/\delta^q$-inexactly smooth where $q = \frac{1-p}{1+p}$ and $\kappa = (q/2)^q \beta^{\frac{2}{1+p}}$. Recall that this is tight up to an absolute constant at most $1.263$. Consequently, we study the class of $L(\delta) = \kappa/\delta^q$-inexactly smooth functions. 

Numerically applying the constructive approach~\eqref{upper_triangular_dual_pep}, we observed that regardless of the parameters $N, \kappa, q, D$, solutions exist with a structured sparsity pattern in $\lambda$. Specifically, only $\lambda_{n-1,n}$ and $\lambda_{\star, n}$ are nonzero. As a result, only the inequalities $\mathcal{Q}_{n-1,n}$ and $\mathcal{Q}_{\star,n}$ and respectively tolerances $\delta_{n-1,n}$ and $\delta_{\star,n}$ are needed to analyze the resulting methods. Once one fixes $\delta_{n-1, n}$ and $\delta_{\star,n}$ and the above sparsity pattern, the resulting $\lambda$ values and induced algorithm are uniquely specified. One arrives at the following method, structurally identical to the Optimized Gradient Method of~\cite{kim2016optimized}, stated in Algorithm~\ref{alg:IS-OGM}, noting our convention $1/\infty = 0$ when applicable.  


\begin{algorithm}[h]
\caption{Inexactly Smooth Optimized Gradient Method}\label{alg:IS-OGM}
\begin{algorithmic}
\Input $x_0$, iteration budget $N$, tolerance sequence $\{\delta_{n-1,n}\}, \{\delta_{\star,n}\}$
\Initialize  $\tau_0 = \frac{1}{L(\delta_{0,1})}+\frac{1}{L(\delta_{\star,0})}$, $z_1 = x_0-\tau_0 g_0$
\For{$n = 1, \ldots, N$}
    \State $\rho_n =  \begin{cases} 
   \sqrt{\frac{1}{L(\delta_{\star,n})^2}+\frac{4\tau_{n-1}}{L(\delta_{n-1,n})}} & \text{if } n = N  \\[1em]
 \frac{1}{L(\delta_{n,n+1})}+\sqrt{\left(\frac{1}{L(\delta_{n,n+1})}+\frac{1}{L(\delta_{\star,n})}\right)^2+\frac{4\tau_{n-1}}{L(\delta_{n-1,n})}+\frac{4\tau_{n-1}}{L(\delta_{n,n+1})}} & \text{else} 
    \end{cases}$
    \State $\tau_n = \tau_{n-1} + \frac{1}{2L(\delta_{\star,n})}+\frac{\rho_n}{2}$
    \State $x_n = \frac{\tau_{n-1}}{\tau_n}\left(x_{n-1}-\frac{1}{L(\delta_{n-1,n})}g_{n-1}\right) + \frac{\tau_n-\tau_{n-1}}{\tau_n}z_n$
    \State $z_{n+1} = z_n - (\tau_n-\tau_{n-1})g_n$
\EndFor
\end{algorithmic}
\end{algorithm}

All that remains is to pick a set of tolerances $\delta_{n-1,n}$ and $\delta_{\star, n}$, given $L(\delta) = \kappa/\delta^q$-inexact smoothness and $N,D>0$. Alas, from numerical solutions to~\eqref{upper_triangular_dual_pep}, an analytic formula remained elusive. However, as $N$ grew, numerical values approached limits 

\begin{equation}\label{eq:optimal_delta_OGM} 
    \delta_{n-1,n} = \left(\frac{q\kappa D^2}{(q+1)^2(N+1)}\right)^\frac{1}{q+1}{n^{-\frac{2}{q+1}}}, \qquad
    \delta_{\star, n} = 0.
\end{equation}
The following theorem proves a convergence guarantee for these choices. Our rates match the lower bounding theory cited in~\cite{nemierovski_optimal_HS} in terms of $N, \kappa, q, D$ up to constants. Moreover, the coefficient of our convergence rate improves upon prior work~\cite{Nesterov_2014}. 

\begin{theorem}\label{thm:UOGM-k/delta-q}
    For any $N\geq 1, D>0$ and $L(\delta)=\kappa/\delta^q$ and any convex $L(\cdot)$-inexactly smooth function $f$ with minimizer $x_\star$ satisfying $\|x_0-x_\star\|\leq D$, Algorithm~\ref{alg:IS-OGM} with tolerance sequence $\delta$ as in~\eqref{eq:optimal_delta_OGM} is guaranteed to have 
    $$f(x_N)-f(x_\star) \leq \frac{\frac{1}{2}D^2+\sigma_N}{\tau_N} \leq \left(\frac{(q+1)^\frac{q-1}{q+1}}{q^\frac{q}{q+1}}+o(1)\right)\frac{\kappa^\frac{1}{q+1}D^\frac{2}{q+1}}{(N+1)^\frac{2-q}{q+1}}.$$
    where $\sigma_N = \sum_{i=1}^N \tau_{i-1}\delta_{i-1,i}$.
    In particular, if $f$ is $(\beta, p)$-\Holder smooth, Algorithm~\ref{alg:IS-OGM} with suitable choices of $\delta$ has
    $$f(x_N)-f(x_\star) \leq \left(\frac{\left(p+1\right)^{p}}{2^{\left(\frac{p+1}{2}\right)}}+o(1)\right)\frac{\beta D^{1+p}}{(N+1)^{\frac{1+3p}{2}}}.$$ 
\end{theorem}
The proof of \cref{thm:UOGM-k/delta-q}, deferred to \cref{sec: appendix algo design}, is a direct generalization of the existing inductive convergence analyses of OGM to include tolerances $\delta$.

\cref{fig:Holder-optimized-plots} presents two numerical results, fixing $\kappa = D = 1$ for ease. The first shows the PEP value~\eqref{upper_triangular_dual_pep} for the GFOM and our asymptotic fit of this method. Their worst-case performance is quite similar, even for small $N$. The second plot shows the effective coefficient of these guarantees for varied $q$, converging to our asymptotically optimized coefficient as $N$ grows. These motivate the following conjecture on the optimality of our leading coefficient.

\begin{conjecture}
     For any $\kappa, D > 0$, $q \in [0,1]$, $N, d \in \N$ with $d \geq N+2$, and any starting point $x_0 \in \R^d$, there exists a convex, $L(\delta) = \kappa/\delta^q$-inexactly smooth $f : \R^d \to \R$ with $\|x_0 -x_\star\| \leq D$ and subgradient oracle such that for any subgradient span method, $$f(x_N)-f(x_\star) \geq \left(\frac{(q+1)^\frac{q-1}{q+1}}{q^\frac{q}{q+1}}+o(1)\right)\frac{\kappa^\frac{1}{q+1}D^\frac{2}{q+1}}{(N+1)^\frac{2-q}{q+1}}.$$
\end{conjecture}

 
\pgfplotstableread{
N   obj
10       0.03607606522306979
15       0.02188720385815014
20       0.015181561994292351
25       0.01137143724597109
30       0.008954188126953484
35       0.007303156608937258
40       0.006114099233768454
45       0.005222859327582777
50       0.00453364249210271
55       0.003987128160370718
60       0.003544746859736119
65       0.003180431454422885
70       0.0028759859706231463
75       0.002618363249330898
80       0.0023979630911128897
85       0.0022075981756154005
90       0.0020417754282686367
95       0.0018962376923454243
100      0.001767635515735061
105      0.0016533113861716498
110      0.00155111287049945
115      0.001459294637123738
120      0.0013764242538472906
125      0.0013013118013381541
130      0.0012329638814592215
135      0.001170560930165336
140      0.001113378408308826
145      0.0010608285030511485
150      0.0010123965230012628
155      0.0009676391394550165
160      0.0009261710797844711
165      0.0008876584537710389
170      0.0008518137888056395
175      0.0008183825425070771
180      0.0007871416825691518
185      0.000757892023073369
190      0.0007304623335865971
195      0.0007046844877979818
200      0.0006804330640150266
}\objdataA
 
\pgfplotstableread{
N   obj
10       0.09722710926696515
15       0.06726688319304523
20       0.051449523034882705
25       0.04166504581576696
30       0.03501215673841935
35       0.03019377656236076
40       0.026542562730304347
45       0.023679996415787375
50       0.02137531142407189
55       0.0194798154817711
60       0.017893375067646704
65       0.016546055223357375
70       0.015387573381117122
75       0.014380797899727413
80       0.013497752493348244
85       0.012716947238136805
90       0.012021577001619498
95       0.01139835559119138
100      0.010836599898223307
105      0.01032764583864734
110      0.009864370716710105
115      0.009440888443554203
120      0.009052295831677843
125      0.00869443659585308
130      0.00836380376171372
135      0.008057409877442696
140      0.007772677516073676
145      0.0075073892787972584
150      0.007259615506191478
155      0.007027685010143156
160      0.006810119203532744
165      0.00660562287858223
170      0.0064130533792747975
175      0.0062313968173314535
180      0.006059751574725368
185      0.005897310960775403
190      0.005743385321203749
195      0.005597246377502984
200      0.005458393495682208
205      0.005326225359061793
210      0.0052003507533494725
215      0.005080253877628551
220      0.004965597056674481
225      0.004856002845101518
230      0.004751142400369393
235      0.0046507171013202495
240      0.004554449360719277
245      0.004462085102043978
250      0.004373422239945994
255      0.004288193549178311
260      0.004206193655386034
265      0.004127443941018605
270      0.0040513053114006325
275      0.003978092637058549
280      0.003907529249626042
285      0.0038393379027879837
290      0.0037737940992054898
295      0.0037099161655526998
300      0.0036484625070387893
305      0.0035887980638285933
310      0.0035311971486414635
315      0.0034754405837002875
320      0.003421587783779369
325      0.0033693090658185806
330      0.0033183110608451406
335      0.0032689490129312075
340      0.0032211321393521295
345      0.003174846372192555
350      0.003129735763125378
355      0.0030857944412882145
360      0.003043163987335758
365      0.003001573445193448
370      0.002961044947684601
375      0.002922079684668388
380      0.0028834675454799505
385      0.0028464431507065637
390      0.0028099291786067744
395      0.0027744744796840627
400      0.0027399422730522496
}\objdataB
 
\pgfplotstableread{
q   coeff
0.0000000100     0.9141923517475232
0.0200000098     0.9744228021225132
0.0400000096     1.0097940144172552
0.0600000094     1.035726393113661
0.0800000092     1.0556012847298692
0.1000000090     1.0710755344573422
0.1200000088     1.0831803790754946
0.1400000086     1.092610699933676
0.1600000084     1.0998748300078756
0.1800000082     1.1053594413868166
0.2000000080     1.1093685438335241
0.2200000078     1.1121420131008781
0.2400000076     1.1138800894214003
0.2600000074     1.1147417763524952
0.2800000072     1.1148651173406028
0.3000000070     1.1143600812703414
0.3200000068     1.1133241985454196
0.3400000066     1.111836237233954
0.3600000064     1.1099655913478952
0.3800000062     1.1077709100734057
0.4000000060     1.1053009937256848
0.4200000058     1.1026001320262233
0.4400000056     1.0997039137482345
0.4600000054     1.0966443311182783
0.4800000052     1.0934489782918742
0.5000000050     1.090140575475997
0.5200000048     1.0867396218837526
0.5400000046     1.083262744235772
0.5600000044     1.0797250624198607
0.5800000042     1.0761389764101412
0.6000000040     1.0725145671410536
0.6200000038     1.0688612075210013
0.6400000036     1.0651864330478482
0.6600000034     1.0614960464016492
0.6800000032     1.0577953327111984
0.7000000030     1.054088243804925
0.7200000028     1.0503784786553667
0.7400000026     1.046668676903808
0.7600000024     1.0429612465333296
0.7800000022     1.0392586729318025
0.8000000020     1.03556367958854
0.8200000018     1.03187951891776
0.8400000016     1.0282106542603795
0.8600000014     1.0245632973195926
0.8800000012     1.0209455312474272
0.9000000010     1.017366903910573
0.9200000008     1.0138234505628752
0.9400000006     1.0103144711720056
0.9600000004     1.0068439456008564
0.9800000002     1.003408736980384
1.0000000000     0.9999998535121649
}\nfifty
 
\pgfplotstableread{
q   coeff
0.0000000100     0.9630841717346035
0.0200000098     1.0228789033448116
0.0400000096     1.0565317435761992
0.0600000094     1.0803323599737837
0.0800000092     1.097826122916489
0.1000000090     1.1108573454347836
0.1200000088     1.1205118827393559
0.1400000086     1.127529268705006
0.1600000084     1.1324459948306025
0.1800000082     1.135639578892113
0.2000000080     1.1374497418383975
0.2200000078     1.1381224590714998
0.2400000076     1.1378549949116878
0.2600000074     1.1368054525056257
0.2800000072     1.1351100974429111
0.3000000070     1.1328893694102027
0.3200000068     1.1302302753104754
0.3400000066     1.127211964599365
0.3600000064     1.123899454309078
0.3800000062     1.1203471931088065
0.4000000060     1.1166041117260268
0.4200000058     1.112707657330969
0.4400000056     1.1086937723803505
0.4600000054     1.104589351470592
0.4800000052     1.1004180514471027
0.5000000050     1.0962019104349876
0.5200000048     1.0919577122898012
0.5400000046     1.0877010980694102
0.5600000044     1.083443782446481
0.5800000042     1.0791963000043274
0.6000000040     1.0749674711261694
0.6200000038     1.070763839429616
0.6400000036     1.066592509592538
0.6600000034     1.062457907128591
0.6800000032     1.0583642537650348
0.7000000030     1.0543138452461123
0.7200000028     1.0503100558391654
0.7400000026     1.0463546318560548
0.7600000024     1.0424489987085734
0.7800000022     1.0385944431722753
0.8000000020     1.0347923751327637
0.8200000018     1.0310440245716073
0.8400000016     1.027351079340901
0.8600000014     1.023715825249151
0.8800000012     1.0201411254769788
0.9000000010     1.016630075067264
0.9200000008     1.013181483646611
0.9400000006     1.0097947788333448
0.9600000004     1.0064706447786163
0.9800000002     1.00320701672134
1.0000000000     0.999999521281315
}\nonefifty
 
\pgfplotstableread{
q   coeff
0.0000000100     0.9854622924657518
0.0200000098     1.0452241460909875
0.0400000096     1.078261837317742
0.0600000094     1.100560185965657
0.0800000092     1.1166989301455055
0.1000000090     1.1285297346234622
0.1200000088     1.137071312653261
0.1400000086     1.1430779227450707
0.1600000084     1.1467893094651287
0.1800000082     1.1490391778831437
0.2000000080     1.1498424679937354
0.2200000078     1.149522373813483
0.2400000076     1.1483418685690883
0.2600000074     1.1464384963568426
0.2800000072     1.1439436229904818
0.3000000070     1.1409708310555782
0.3200000068     1.1375999905040852
0.3400000066     1.133899511552952
0.3600000064     1.1299621045938648
0.3800000062     1.1258205162409245
0.4000000060     1.1215193826564342
0.4200000058     1.1171098276323288
0.4400000056     1.112613271253866
0.4600000054     1.1080575201011007
0.4800000052     1.1034680272566706
0.5000000050     1.098860923037509
0.5200000048     1.0942594406902142
0.5400000046     1.0896687243019807
0.5600000044     1.0851040690092493
0.5800000042     1.0805755218467816
0.6000000040     1.0760924051335174
0.6200000038     1.0716569433605698
0.6400000036     1.067274609375309
0.6600000034     1.0629535675375692
0.6800000032     1.05869349047597
0.7000000030     1.0545007102244055
0.7200000028     1.0503722381015932
0.7400000026     1.046313253546301
0.7600000024     1.0423219708522968
0.7800000022     1.038406464276654
0.8000000020     1.0345555863299256
0.8200000018     1.030777051004371
0.8400000016     1.0270716202333459
0.8600000014     1.0234371665524984
0.8800000012     1.0198718948373355
0.9000000010     1.016382331874672
0.9200000008     1.0129717126434545
0.9400000006     1.0096221663155118
0.9600000004     1.0063446097570745
0.9800000002     1.0031392078273607
1.0000000000     1.0000016566589285
}\nfourfifty
 
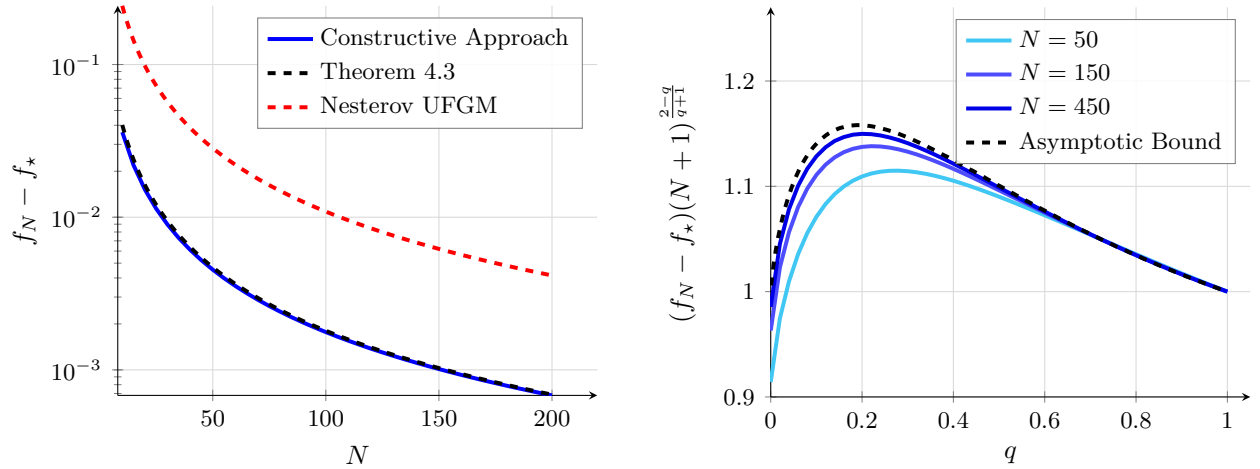
\begin{figure}[htbp]
\centering
 
\begin{subfigure}[b]{0.48\textwidth}
\centering
\begin{tikzpicture}
\begin{axis}[
    axis lines=left,
    width=\textwidth,
    height=0.85\textwidth,
    ymode=log,
    xlabel={$N$},
    ylabel={$f_N-f_\star$},
    xmin=8, xmax=220,
    grid=major,
    grid style={gray!30},
    legend style={
        at={(0.97,0.97)},
        anchor=north east,
        font=\footnotesize,
        draw=black!50,
        cells={anchor=west},
        fill opacity = 0.9,
        text opacity = 1,
        row sep=1pt,
    },
    tick label style={font=\footnotesize},
    label style={font=\small},
]
 
\addplot[color=blue, solid, line width=1.5pt, mark=none]
    table [x=N, y=obj] {\objdataA};
\addlegendentry{Constructive Approach}
 
\addplot[color=black, dashed, line width=1.5pt, mark=none,
    domain=10:200, samples=200]
    {1.1541599247 / (x+1)^1.4};
\addlegendentry{\cref{thm:UOGM-k/delta-q}}
 
\addplot[color=red, dashed, line width=1.5pt, mark=none,
    domain=10:200, samples=200]
    {6.9644045064 / (x+1)^1.4};
\addlegendentry{Nesterov UFGM}
 
\end{axis}
\end{tikzpicture}
\end{subfigure}%
\hfill
\begin{subfigure}[b]{0.48\textwidth}
\centering
\begin{tikzpicture}
\begin{axis}[
    axis lines=left,
    width=\textwidth,
    height=0.85\textwidth,
    xlabel={$q$},
    ylabel={$(f_N-f_\star){(N+1)^\frac{2-q}{q+1}}$},
    xmin=0, xmax=1.05,
    ymin=0.9, ymax=1.27,
    grid=major,
    grid style={gray!30},
    legend style={
        at={(0.97,0.97)},
        anchor=north east,
        font=\footnotesize,
        draw=black!50,
        cells={anchor=west},
        fill opacity = 0.9,
        text opacity = 1,
        row sep=1pt,
    },
    tick label style={font=\footnotesize},
    label style={font=\small},
]
 
\addplot[color=cyan!60!white, solid, line width=1.5pt, mark=none]
    table [x=q, y=coeff] {\nfifty};
\addlegendentry{$N = 50$}
 
\addplot[color=blue!70!white, solid, line width=1.5pt, mark=none]
    table [x=q, y=coeff] {\nonefifty};
\addlegendentry{$N = 150$}
 
\addplot[color=blue!90!black, solid, line width=1.5pt, mark=none]
    table [x=q, y=coeff] {\nfourfifty};
\addlegendentry{$N = 450$}
 
\addplot[color=black, dashed, line width=1.5pt, mark=none,
    domain=0.001:1, samples=200]
    {(x+1)^((x-1)/(x+1)) / x^(x/(x+1))};
\addlegendentry{Asymptotic Bound}
 
\end{axis}
\end{tikzpicture}
\end{subfigure}
 \caption{The first plot compares the numerical output of the constructive approach~\eqref{constructive_approach_PEP} for $L(\delta) = 1/\delta^q$ with $q = 0.25$ with Nesterov~\cite[Theorem 3]{Nesterov_2014} and \cref{thm:UOGM-k/delta-q}. The second displays the conjectured asymptotic tightness of this leading coefficient.}\label{fig:Holder-optimized-plots}
\end{figure}

\subsection{An Optimized Universal Method for Inexactly Smooth Problems}\label{subsec:OBL}

As a final algorithm, we design an optimized universal and parameter-free method applicable for any $L(\cdot)$ satisfying our assumptions. Universal and parameter-free algorithms offer guarantees for a range of problem instances while requiring no input parameters dependent on their structure (i.e., the function $L(\cdot)$ or optimized tolerances $\delta$). For \Holder smooth convex minimization, such a method (UFGM) was previously designed by~\cite{Nesterov_2014}, generalizing Nesterov's fast gradient method to be universal. Here we similarly extend the OBL method of~\cite{OBL} to handle inexactly smooth functions.  

\begin{algorithm}[t]
\caption{Universal Optimized Backtrackable Linesearch method (UOBL)}\label{alg:UOBL}
\begin{algorithmic}
\Input $x_0$, $L_0 > 0$, iteration budget $N$, and target accuracy $\eps > 0$
\Initialize $\tau_0 = 1/L_0$, $z_1 = x_0-\tau_0g_0$
\For{$n = 1, \ldots, N$}
    \State Find the smallest $i \geq 0$ such that, with $L_n = 2^i L_{n-1}$, the values
    \State \quad $\tau_n = \begin{cases} \tau_{n-1}+\sqrt{\tau_{n-1}/{L_{n}}}  & \text{if } n = N \\[0.5em] \tau_{n-1}+\dfrac{1+\sqrt{1+8\tau_{n-1}L_{n}}}{2L_{n}} & \text{else} \end{cases}$
    \State \quad $x_n = \frac{\tau_{n-1}}{\tau_n}\left(x_{n-1}-\frac{1}{L_n}g_{n-1}\right)+\frac{\tau_n-\tau_{n-1}}{\tau_n}z_n$
      \State satisfy $f(x_{n-1})-f(x_n)-\inner{g_n}{x_{n-1}-x_n}-\frac{1}{2L_n}\|g_{n-1}-g_n\|^2 + \frac{\tau_n-\tau_{n-1}}{\tau_{n-1}}\cdot\frac{\eps}{2} \geq 0$
    \State $z_{n+1} = z_n - (\tau_n-\tau_{n-1})g_n$
  
\EndFor
\end{algorithmic}
\end{algorithm}

The UOBL method is defined in Algorithm~\ref{alg:UOBL}. The following theorem establishes its convergence guarantee, including an accumulated error term $\Delta_N$ identical to~\cite[Theorem 6]{OBL}, which grows only at each of the logarithmically many backtracking steps.

\begin{theorem}\label{thm:UOBL convergence}
    For any $N\geq 1, D,\eps>0$ and $L(\cdot)$ satisfying our assumptions, consider a convex $L(\cdot)$-inexactly smooth $f$ with minimizer $x_\star$ satisfying $\|x_0-x_\star\| \leq D$.  
    Then Algorithm~\ref{alg:UOBL} is guaranteed to have
  $$f(x_N)-f(x_\star)  \leq \frac{\max\left\{L_0, 2L\left(\frac{\eps}{\sqrt{2}N}\right)\right\}(D^2+\Delta_N)}{N^2}+\frac{\eps}{2}$$
    where $\Delta_N = \sum_{i=1}^N {\tau_{i-1}}\left(\frac{1}{L_{i-1}}-\frac{1}{L_i}\right)\|g_{i-1}\|^2$.
    In particular, if $f$ is $L(\delta) = \kappa/\delta^q$-inexactly smooth, then for $L_0$ sufficiently small, Algorithm~\ref{alg:UOBL} has a big-O optimal convergence rate of
    $$f(x_N)-f(x_\star) \leq \frac{2^{1+\frac{q}{2}}\kappa\left(D^2 + \Delta_N\right) }{N^{2-q}\eps^{q}}+\frac{\eps}{2}.  $$
\end{theorem}
\begin{proof}
Consider inexact tolerances $\delta_{i-1,i}:=\frac{\tau_{i}-\tau_{i-1}}{\tau_{i-1}}\frac{\eps}{2}$. The proof of our convergence rate follows inductively, maintaining nonnegativity of the following quantities: for $n=0,\dots,N$, define
\begin{align*}H_n &= \tau_n(f_\star-f_n+\frac{1_{n < N}}{2L_n}\|g_n\|^2)+\frac{1}{2}\|x_0-x_\star\|^2-\frac{1}{2}\|z_{n+1}-x_\star\|^2 \\ &\qquad+\sum_{i=1}^n \tau_{i-1}\delta_{i-1,i}+\sum_{i=1}^{n} \frac{\tau_{i-1}}{2}\left(\frac{1}{L_{i-1}}-\frac{1}{L_{i}}\right)\|g_{i-1}\|^2, \end{align*}
where $1_{n < N} = 1$ when $n < N$ and $0$ at $n=N$. 

With initialization $\tau_0=\frac{1}{L_0}$ for some $L_0 > 0$ and $z_1 = x_0-\tau_0g_0$, the base case holds that $H_0 = \tau_0 \mathcal{C}_{\star,0} \geq 0$ (by convexity of $f$) with $\mathcal{C}_{\star, 0}$ defined below. For $n=1,\dots,N$, the induction is maintained by observing the identity
$$H_{n}=H_{n-1}+\tau_{n-1}\tilde{\mathcal{Q}}_{n-1,n,\delta_{n-1,n}}+(\tau_n-\tau_{n-1})\mathcal{C}_{\star,n}$$
with
\begin{align*}\tilde{\mathcal{Q}}_{n-1,n,\delta_{n-1,n}} &:= f_{n-1}-f_n - \inner{g_n}{x_{n-1}-x_n}-\frac{1}{2L_n}\|g_{n-1}-g_n\|^2 + \delta_{n-1,n},\\
\mathcal{C}_{\star, n} &:= f_\star - f_n - \inner{g_n}{x_\star - x_n}.\end{align*}
This identity shows that $H_n$ is the sum of three quantities that are nonnegative by definition, and hence $H_n$ is also nonnegative. Then, a convergence rate follows from rearranging $H_N\geq 0$ as
\begin{equation}\label{eq:OBL bound}f_N-f_\star \leq \frac{\frac{1}{2}D^2 + \sum_{i=1}^N \frac{\tau_{i-1}}{2}\left(\frac{1}{L_{i-1}}-\frac{1}{L_i}\right)\|g_{i-1}\|^2 }{\tau_N}+\frac{\eps}{2}\end{equation}
by our choices of $\delta_{i-1,i}$ and the nonnegativity of $\tau_0$.

To arrive at our claimed guarantee, we must bound $\tau_N$. For $L_0$ large enough, $L_n = L_0$ for all $n = 1, \dots, N$. Therefore we may bound $\tau_N \geq \frac{N^2}{2L_0}$ by considering the recurrence in~\cite[Corollary 5]{OBL} for fixed $L$. Otherwise, suppose we backtrack at least once. Letting $r_n = \frac{\tau_n-\tau_{n-1}}{\tau_{n-1}}$, note that
$$\tau_n-\tau_{n-1} = \begin{cases} \frac{1+\sqrt{1+8\tau_{n-1}L_n}}{2L_n} & n < N\\
\sqrt{\frac{\tau_{n-1}}{L_n}} & n = N\end{cases} \ \Longrightarrow \ r_n \geq \sqrt{\frac{1}{L_n\tau_{n-1}}}, \quad  \forall n \leq N.$$
Since the algorithm ensures $\tilde{\mathcal{Q}}_{n-1,n,\delta_{n-1,n}}  \geq 0$ for $\delta_{n-1,n}=\frac{(\tau_n-\tau_{n-1})}{\tau_{n-1}}\frac{\eps}{2}$, it follows that $$L_{n} \leq 2L\left(\frac{\eps}{2}\frac{(\tau_n-\tau_{n-1})}{\tau_{n-1}}\right) \leq 2L\left(\frac{\eps}{2\sqrt{L_n\tau_{n-1}}}\right)\leq 2L\left(\frac{\eps}{2\sqrt{L_N\tau_N}}\right)$$
where the first inequality holds by the doubling scheme on $L_n$, the second inequality considers the bound on $r_n$ and the monotonicity of $L(\cdot)$, and the last inequality follows from monotonicity of $L_n\tau_{n-1}$ and $\tau_N$.
Therefore, we conclude the uniform bound $$L_n \leq  2L\left(\frac{\eps}{2\sqrt{L_N\tau_N}}\right), \quad \forall n \leq N.$$

Define as the unique positive solution $L_\eps(\tau) := \left\{\hat{L} > 0 : \hat{L} = 2L\left(\frac{\eps}{2\sqrt{\hat{L}\tau}}\right)\right\},$ which exists by monotonicity of $L(\cdot)$ and convexity of $-1/L(\cdot)$. These same assumptions, along with $L_N \leq 2L\left(\frac{\eps}{2\sqrt{L_N\tau_N}}\right)$ enforce $L_n \leq L_\eps(\tau_N)$ for all $n \leq N$. Therefore $$\tau_N \geq \frac{N^2}{ 4L\left(\frac{\eps}{2\sqrt{L_N\tau_N}}\right)} \geq \frac{N^2}{2L_\eps(\tau_N)}$$ by noting the same recurrence as above that $\tau_N$ has with fixed $L$.
Finally, taking the unique positive solution to $\hat{\tau}_\eps(N) := \left\{\tau > 0 : \tau = \frac{N^2}{2L_\eps(\tau)}\right\},$ we bound $\tau_N \geq \hat{\tau}_\eps(N)$ from our assumptions on $L(\cdot)$. Using $\hat{\tau}_\eps(N)$ and $L_\eps(\hat{\tau}_\eps(N))$ to simplify \eqref{eq:OBL bound} yields the claimed result.

In the case where $f$ is $L(\delta) = \kappa/\delta^q$-inexactly smooth, we can express these functions exactly. Provided that $L_0 < 2^{1+q/2}\kappa\left(\frac{N}{\eps}\right)^{q}$, it holds that
\begin{align*}
    f_N-f_\star &\leq \frac{\frac{1}{2}D^2 + \sum_{i=1}^N \frac{\tau_{i-1}}{2}\left(\frac{1}{L_{i-1}}-\frac{1}{L_i}\right)\|g_{i-1}\|^2 }{\hat{\tau}_\eps(N)}+\frac{\eps}{2}\\
    &=\frac{2^{2+\frac{q}{2}}\kappa\left(\frac{1}{2}D^2 + \sum_{i=1}^N \frac{\tau_{i-1}}{2}\left(\frac{1}{L_{i-1}}-\frac{1}{L_i}\right)\|g_{i-1}\|^2\right) }{N^{2-q}\eps^{q}}+\frac{\eps}{2}
\end{align*}
where
$L_\eps(\tau) = \ \left(\frac{2^{q+1} \kappa \tau^{\frac{q}{2}}}{\eps^q}\right)^\frac{2}{2-q}$ and $\hat{\tau}_\eps(N) = \frac{N^{2-q}\eps^q}{2^{2+\frac{q}{2}}\kappa}$.
\end{proof}

    \paragraph{Acknowledgments.} Benjamin Grimmer was supported as an Alfred P. Sloan Foundation fellow.

    {\small
    \bibliographystyle{plainurl}
    \bibliography{bibliography}
    }
    \appendix

\section{Deferred Proofs on Inexactly Smooth Calculus}\label{sec: appendix L(delta) calc}

\begin{proof}[\bf Proof of \cref{lemma:L(delta) conjugate equivalence}]
($1. \Longrightarrow 2.$) Here we have that for all $x, y$ with $g_x \in \partial f(x), \ g_y \in \partial f(y)$ and all $\delta \geq 0$ $$f(x) \geq f(y) + \inner{g_y}{x-y}+\frac{1}{2L(\delta)}\|g_y-g_x\|^2 -\delta .$$
Writing $g_y = g_x + (g_y-g_x)$ and applying Young's inequality to bound the inner product $\inner{g_y-g_x}{y-x}$ by $\frac{L(\delta)}{2}\|y-x\|^2 + \frac{1}{2L(\delta)}\|g_y-g_x\|^2$, yields the second condition.

\noindent ($2. \Longrightarrow 3.$)  It holds that for any $x,z$ and $g_x \in \partial f(x)$ and any $\delta \geq 0$, \begin{align*}
    -f(z) &\geq f^*(g_x)-\inner{z}{g_x}-\frac{L(\delta)}{2}\|z-x\|^2-\delta,
\end{align*}
where we use the Fenchel-Young equality \eqref{eq:Fenchel-Young}.
In turn, for any $y$ with $g_y \in \partial f(y)$, \begin{align*}
    f^*(g_y) &\geq \inner{z}{g_y}-f(z)\\
    &\geq \inner{z}{g_y}+f^*(g_x)-\inner{z}{g_x}-\frac{L(\delta)}{2}\|z-x\|^2-\delta\\
    &= f^*(g_x)+\inner{x}{g_y-g_x}+\inner{z-x}{g_y-g_x}-\frac{L(\delta)}{2}\|z-x\|^2-\delta.
\end{align*}
Taking the supremum over $z$ gives $f^*(g_y) \geq f^*(g_x) + \inner{x}{g_y-g_x}+\frac{1}{2L(\delta)}\|g_y-g_x\|^2 -\delta.$

\noindent ($3. \Longrightarrow 1.$) It holds with the Fenchel-Young equality \eqref{eq:Fenchel-Young} that $$\inner{g_y}{y}-f(y) = f^*(g_y) \geq \inner{g_x}{x}-f(x)+\inner{x}{g_y-g_x}+\frac{1}{2L(\delta)}\|g_x-g_y\|^2-\delta .$$
Rearranging, swapping the roles of $x$ and $y$ above, and noting this holds for all $x, y$ with respective subgradients $g_x, g_y$ and all $\delta \geq 0$ yields the claimed inexact cocoercive-type condition. 
\end{proof}

\begin{proof}[\bf Proof of \cref{L(delta) scaling calculus}] 
We now prove each of the three claims separately.
    \begin{enumerate}
        \item Noting that $g_x \in \partial f(x) \Longleftrightarrow \alpha g_x \in \partial (\alpha f)(x)$, applying the change of variables $\delta \mapsto \delta/\alpha$ in~\eqref{def: inexactly smooth} and rearranging proves the claim.
        \item Note that for any $g_{Ax-b} \in \partial f(Ax-b)$, by the subgradient chain rule~\cite[Corollary 16.72]{bauschke}, one has $A^T g_{Ax-b} \in \partial (f \circ (A(\cdot) -b))(x)$. Therefore, with the change of variables $x \mapsto Ax-b$, \eqref{def: inexactly smooth} becomes
        \begin{eqnarray*}
            f(Ay-b)
            &\leq f(Ax-b) + \inner{A^Tg_{Ax-b}}{y-x}+\frac{\|A\|_{\mathrm{op}}^2L(\delta)}{2}\|y-x\|^2+\delta.
            \end{eqnarray*}
            
        \item For any fixed $x$, let $z_x \in \argmin_{z} f(x,z)$. Then $F(x) = f(x,z_x)$ and $F(y) = \inf_z f(y,z) \leq f(y,z_x)$. Since $f(x,z)$ is $L(\cdot)$-inexactly smooth it holds for any $(x,z)$, $(y,z')$ and $(g_x, g_z) \in \partial f(x,z)$ and any $\delta \geq 0$ that $$f(y,z') \leq f(x,z) + \inner{g_x}{y-x}+\inner{g_z}{z'-z}+\frac{L(\delta)}{2}(\|x-y\|^2+\|z-z'\|^2) + \delta.$$
        From the subdifferential calculus rule~\cite[Theorem 1]{infprojection}, we know that $\partial F(x) = \{g_x : (g_x, 0) \in \partial f(x,z_x), z_x \in \argmin_z f(x,z)\}$. For any $z_x \in \argmin_z f(x,z)$, and any associated $g_x$ such that $(g_x,0) \in \partial f(x,z_x)$, let  $z = z' = z_x$ and $g_z=0$. Then, the above inequality establishes the inexact smoothness of $F$ as
        $$F(y) \leq f(y,z_x) \leq F(x) + \inner{g_x}{y-x} + \frac{L(\delta)}{2}\|y-x\|^2+\delta.  $$
    \end{enumerate}
\end{proof}

\begin{proof}[\bf Proof of \cref{lemma:L(delta) sums and max calculus}]
    Recall from \cref{lemma:L(delta) conjugate equivalence} that $L_i$-inexact smoothness of $f_i$ ensures that all $x, y$ with $g_x^{(i)} \in \partial f_i(x)$ and $g_y^{(i)}\in \partial f_i(y)$ and $\delta_i \geq 0$ have
    \begin{equation}\label{primal-bound_indexed}f_i(y) \leq f_i(x) + \inner{g_x^{(i)}}{y-x}+\frac{L_i(\delta_i)}{2}\|y-x\|^2 + \delta_i,\end{equation}
    \begin{equation}\label{dual-bound_indexed}f_i^*(g_y^{(i)}) \geq f_i^*(g_x^{(i)}) + \inner{x}{g_y^{(i)}-g_x^{(i)}}+\frac{1}{2L_i(\delta_i)}\|g_x^{(i)}-g_y^{(i)}\|^2 - \delta_i .\end{equation}
    \begin{enumerate}
        \item For any choice of $\delta_i\geq 0$, summing each inequality in \eqref{primal-bound_indexed} with $\sum_{i=1}^m {\delta_i}=\delta$ and utilizing the sum rule~\cite[Theorem 23.8]{rockafeller}
        establishes a suitable quadratic upper bound for the function $\sum_{i=1}^m f_i$. Taking the infimum over all selections of $\delta_i\geq 0$ gives the claimed formula.
        \item Consider the function given by the inner maximum $f^*(g_x) = \max_{i} f_i^*(g_x)$. By~\cite[Theorem 18.5]{bauschke}, letting $\mathcal{I}^*(g_x) = \{i : f^*(g_x) = f_i^*(g_x)\}$ the subdifferential of this maximum is given by $$\partial f^*(g_x) = {\mathrm{conv}} \bigcup_{i \in \mathcal{I}^*(g_x)} \partial f_i^*(g_x) .$$
        Fix any $g_x$ with $x \in \partial f^*(g_x)$, which the above formula guarantees must take the form $x = \sum_{i \in \mathcal{I}^*(g_x)} \alpha_ix_i$ with $x_i \in \partial f_i^*(g_x)$. Then for any $g_y$, \begin{align*}
            f^*(g_y) &\geq \sum_{i \in \mathcal{I}^*(g_x)} \alpha_i f_i^*(g_y)\\
            &\geq \sum_{i \in \mathcal{I}^*(g_x)} \alpha_i(f_i^*(g_x) + \inner{x_i}{g_y-g_x}+\frac{1}{ 2L_i(\delta)}\|g_y-g_x\|^2 -\delta)\\
            &\geq f^*(g_x) + \inner{x}{g_y-g_x}+\frac{1}{2\max_i L_i(\delta)}\|g_y-g_x\|^2 -\delta .
        \end{align*}
        Hence $f^*$ is $1/\max_i L_i(\delta)$-inexactly strongly convex and so \cref{lemma:L(delta) conjugate equivalence} gives the result.
        \item Summing the bounds \eqref{dual-bound_indexed} with any selection of $\delta_i\geq 0$ verifies the $\sum 1/L_i(\delta_i)$-inexact strong convexity of $\sum f_i^*$. Taking the supremum over choices of $\delta_i\geq 0$ with $\sum \delta_i = \delta$ tightens this bound, from which the claimed inexact smoothness follows again by \cref{lemma:L(delta) conjugate equivalence}.

    \end{enumerate}
\end{proof}

\begin{proof}[\bf Proof of \cref{composing with abs to norm}]
    For any $x, y \in \R^d$, $g_x \in \partial f(x)$, and $\delta \geq 0$, it suffices to show that the following $S_f$ function is nonnegative. Namely, 
    $$0 \leq S_f(x,y,g_x,\delta) := f(x) + \inner{g_x}{y-x} + \frac{L(\delta)}{2}\|y-x\|^2 + \delta - f(y).$$
    Since $h$ is convex with full domain, we can apply a subgradient chain rule~\cite[Corollary 16.72]{bauschke} to $f=h\circ\|\cdot\|$. From this, any subgradient $g_x\in\partial f(x)$ can be decomposed into the form
    $$g_x = \zeta u, \quad \zeta \in \partial h(\|x\|), \quad u \in \partial \|\cdot\|(x).$$
    Consider minimizing $S_f$ over all values $y$ with some fixed norm, $\|y\| = s$. Note that if $x \ne 0$ then one has $u = x/\|x\|$ whereas if $x = 0$, one can have any $u \in B(0,1)$. In either case, since $S_f$ is a simple quadratic with respect to $y$ on this sphere, $S_f$ must minimize over $y$ somewhere collinear with $g_x$. From this, the result follows as $$S_f(x,y,g_x, \delta) \geq \min \{h(\|x\|)+\zeta\cdot(\pm \|y\|-\|x\|)+\frac{L(\delta)}{2}(\pm\|y\|-\|x\|)^2+\delta-h(\|y\|) \}$$ where the above is nonnegative from $h(|\cdot|)$ being $L(\cdot)$-inexactly smooth.
\end{proof}

\section{Deferred Proofs for Performance Estimation Theory}\label{sec: appendix PEP}

\begin{proof}[\bf Proof of \cref{lemma:inexactly_smooth_construction}]
We first show $k  = \inf_{\delta \geq 0} \left\{\frac{L\left(\delta\right)+\sqrt{L\left(\delta\right)^{2}+2\delta L\left(\delta\right)/R^2}}{2}\right\}$ is strictly positive. For $\delta \in [0,1]$, the term in the infimum is greater than $L(\delta) \geq L(1)$ by monotonicity.  For $\delta\in [1,\infty)$, we observe that the expression inside the infimum is also larger than
$\frac{\sqrt{\delta L(\delta)}}{R\sqrt{2}}$. To uniformly lower bound this on $[1,\infty)$, notice that the concavity of $1/L(\cdot)$ gives $\frac{1}{L(1)} \geq \frac{1}{\delta}\frac{1}{L(\delta)}+\left(1-\frac{1}{\delta}\right)\frac{1}{L(0)}.$ Together with $1/L(1) \geq 1/L(0) \geq 0$, where we utilize the convention $1/\infty = 0$, yields
$$
\frac{1}{L(\delta)}\leq\frac{1}{L(0)}+\left(\frac{1}{L(1)}-\frac{1}{L(0)}\right)\delta, \qquad \delta\geq 1.
$$
Therefore, $\delta L(\delta)\geq \delta/\left(\frac{1}{L(0)}+\left(\frac{1}{L(1)}-\frac{1}{L(0)}\right)\delta\right).$ The right-hand side is increasing in $\delta$, so $\delta L(\delta)\geq L(1)$ for all $\delta\geq 1$. Together,
$
k \geq  \min\left\{L(1), \, \frac{\sqrt{L(1)}}{R\sqrt{2}}\right\} > 0. 
$

    To verify $f$ is $L(\cdot)$-inexactly smooth, we check the conditions directly. Since the function $h(x)$ is extended linearly outside of the ball with radius $R$, as a consequence of \cref{composing with abs to norm}, it suffices to only verify for $\|x\|, \|y\| \leq R$.  We note that for given $\|x\| \leq R$, $h(x) = \frac{k}{2}\|x\|^2$ and $\nabla h(x) = kx$. Recall the cocoercive-like condition from \cref{lemma:L(delta) conjugate equivalence} for being $L(\cdot)$-inexactly smooth states that for all $x, y$ and $\delta\geq 0$, $h(y) \geq h(x) + \inner{\nabla h(x)}{y-x}+\frac{1}{2L(\delta)}\|\nabla h(x)-\nabla h(y)\|^2 - \delta$.
    In this case, this simplifies to all $\|x\|, \|y\| \leq R$ and $\delta\geq 0$ having
    \begin{align*}
       0 \geq \frac{k^2}{2L(\delta)}\|y-x\|^2 - \frac{k}{2}\|y-x\|^2 - \delta  .
    \end{align*}
  Our choice of $k =  \inf_{\delta \geq 0} \left\{\frac{L\left(\delta\right)+\sqrt{L\left(\delta\right)^{2}+2\delta L\left(\delta\right)/R^2}}{2}\right\}$ is exactly the maximum value of $k$ for which this holds.
  Therefore $h$ is $L(\cdot)$-inexactly smooth. To conclude, we note that so long as $\|Q\|_{\mathrm{op}} \leq 1$, then by \cref{L(delta) scaling calculus}, $h \circ Q^{1/2}$ is $L(\cdot)$-inexactly smooth as well.
\end{proof}

\begin{proof}[\bf Proof of \cref{lemma:slater point}]
     We construct a Slater point of a similar form to~\cite[Theorem 6]{taylor2017smooth}. In particular, under the Gram reformulation of \eqref{eq:Inexact-Primal-infinite-interpolation}, we show the existence of a strictly feasible point with $G \succ 0$ and $\mathcal{Q}_{i,j} > 0$. 
    
    For $L(\cdot)$ satisfying our assumptions, consider the constructed $L(\cdot)$-inexactly smooth function $h(x)$ in \cref{lemma:inexactly_smooth_construction} with  $R = D(1+\frac{L(1)+\sqrt{L(1)^2+2L(1)/D^2}}{2}\|W\|_\infty)^N.$ Define $$Q = \frac{0.5}{2+2\cos\left(\frac{\pi}{d+1}\right)}\begin{bmatrix}
        2 & 1 & 0 & \dots & 0 \\
        1 & 2 & 1 & \dots & 0 \\
        0 & 1 & 2 & \dots & 0\\
        \vdots & \ddots & \ddots & \ddots & \vdots \\
        0 & 0 & 0 & \dots & 2
    \end{bmatrix}, \quad f(x) = h(Q^{1/2}x).$$
    By the calculus rules in~\cref{L(delta) scaling calculus}, $f$ is $0.5 L(\cdot)$-inexactly smooth and therefore $f$ is $L(\cdot)$-inexactly smooth. Furthermore, $f$ has Hessian $kQ$ for all $\|x\|_{Q} \leq R$ with $k := \inf_{\delta \geq 0} \left\{\frac{L\left(\delta\right)+\sqrt{L\left(\delta\right)^{2}+2\delta L\left(\delta\right)/R^2}}{2}\right\}$ and minimizer $x_\star = 0$. Choosing $x_0 = D e_1$, we show by induction that $\|x_n\|_Q \leq R$ for all $n \leq N$. The base case holds since $\|Q\|_{\mathrm{op}} \leq 1$ and $D \leq R$. Then suppose $\|x_i\|_Q \leq D(1+\frac{L(1)+\sqrt{L(1)^2+2L(1)/D^2}}{2}\|W\|_\infty)^i$ for all $i \leq n-1$. Consequently, \begin{align*}\|x_n\|_Q \leq \|x_0\|_Q + \|\sum_{i=0}^{n-1} W_{n, i}g_i\|_Q
    &\leq D + k\|W\|_\infty D(1+k\|W\|_\infty)^{n-1}  \\
    &\leq D(1+\frac{L(1)+\sqrt{L(1)^2+2L(1)/D^2}}{2}\|W\|_\infty)^n\end{align*}
    where the first inequality uses the triangle inequality, the second substitutes $g_i = kQx_i$, applies the induction hypothesis and submultiplicativity of the matrix norm, and the last uses the definition of $k$, the bound $D \leq R$, and nonnegativity of the norm.
    
    Therefore, $g_n = kQx_n$ for all $n \leq N$. Furthermore, under the assumption $W_{i,i-1} \ne 0$ (see~\cite[Theorem 6]{taylor2017smooth}), it holds that $P = [x_0 \ | \ g_0 \ |\  g_1 \ |\  \dots\ | \ g_N]$ is upper triangular with non-zero diagonals and $G=P^TP \succ 0$. Consequently, each gradient is distinct and $f$ belongs to a stricter class than $L(\cdot)$-inexactly smooth functions, having $\mathcal{Q}_{i,j}(F,G)>0$. With all non-affine constraints strictly feasible, this is a Slater point. 
\end{proof}

\section{Deferred Proofs for Algorithm Design}\label{sec: appendix algo design}

\begin{proof}[\bf Proof of the derivation for \eqref{upper_triangular_dual_pep}]
Recall the definition of $\pAlg_{interp}$ given in~\eqref{constructive_approach_PEP}. We can reformulate this by applying the zero equality constraints and noting that $x_\star \in \mathrm{span}\{g_i\}$ implies $\|x_\star\|^2 = \sum_{i=0}^N \frac{\inner{g_i}{x_\star}^2}{\|g_i\|^2} \leq D^2$. Doing so, $\pAlg_{interp}$ equals
\begin{equation}\label{simplified_PEP_delta}
\begin{cases}
    \displaystyle\max_{f_i, \|g_i\|^2, \inner{g_i}{x_j}} & f_N - f_\star \\
    \mathrm{s.t.} & f_i - f_j - \langle g_j, x_i \rangle - \sup_{\delta_{i,j} \geq 0}\left\{\frac{\|g_i\|^2 + \|g_j\|^2}{2L(\delta_{i,j})} - \delta_{i,j}\right\} \geq 0, \quad j < i \\
    & f_i - f_j - \sup_{\delta_{i,j} \geq 0}\left\{\frac{\left(\|g_i\|^2 + \|g_j\|^2\right)}{2L(\delta_{i,j})} - \delta_{i,j}\right\} \geq 0, \quad i < j  \\
    & f_\star - f_i - \langle g_i, x_\star \rangle - \sup_{\delta_{\star, i} \geq 0}\left\{\frac{1}{2L(\delta_{\star,i})}\|g_i\|^2 - \delta_{\star,i}\right\} \geq 0, \\
    & f_i - f_\star - \sup_{\delta_{i,\star} \geq 0}\left\{\frac{1}{2L(\delta_{i,\star})}\|g_i\|^2 - \delta_{i, \star}\right\} \geq 0, \\
    & \sum_{i=0}^N \frac{\inner{g_i}{x_\star}^2}{\|g_i\|^2} \leq D^2
\end{cases}
\end{equation}
with $i, j$ ranging $0, \dots, N$.
The associated dual program is then

\begin{equation}
\dAlg_{interp} = \begin{cases}
   \displaystyle \inf_{\delta_{i,j} \geq 0} \min_{\lambda_{i,j}, \lambda_{\star,i}, \lambda_{i,\star}, \nu} & \Phi(\lambda, \delta, \nu) \\
    \mathrm{s.t.} & \sum_{i=j+1}^{N} \lambda_{j,i} - \sum_{i=0}^{j-1} \lambda_{i,j} = \lambda_{\star,j}-\lambda_{j,\star}, \quad j < N \\
    & \sum_{i=0}^{N-1} \lambda_{i,N} = \sum_{i=0}^{N-1} \lambda_{\star,i}-\lambda_{i,\star} \\
    & \sum_{i=0}^{N} \lambda_{\star,i}-\lambda_{i,\star} = 1 \\
    & \lambda_{i,j} \geq 0, \ \lambda_{\star,i} \geq 0, \ \lambda_{i,\star} \geq 0, \ \nu \geq 0 .
\end{cases}
\end{equation}
where we derive the constraints in the dual program by considering the linear components of the Lagrangian and consider the dual problem $\Phi(\lambda, \delta, \nu) = $\begin{align*}
   \max_{\|g_i\|^2, \inner{g_j}{x_i}}  \nu D^2 
    &+ \textstyle\sum_{0 \leq j < i \leq N} \lambda_{i,j} \left[ \delta_{i,j} - \langle g_j, x_i \rangle - \frac{1}{2L(\delta_{i,j})}\left(\|g_i\|^2 + \|g_j\|^2\right) \right] \\
    & + \textstyle\sum_{0 \leq i < j \leq N} \lambda_{i,j} \left[ \delta_{i,j} - \frac{1}{2L(\delta_{i,j})}\left(\|g_i\|^2 + \|g_j\|^2\right) \right] \\
    & + \textstyle\sum_{i=0}^{N} \lambda_{\star,i} \left[ \delta_{\star,i} - \langle g_i, x_\star \rangle - \frac{1}{2L(\delta_{\star,i})}\|g_i\|^2 \right] \\
    & + \textstyle\sum_{i=0}^N \lambda_{i,\star}\left[\delta_{i,\star}-\frac{1}{2L(\delta_{i,\star})}\|g_i\|^2\right]- \nu \sum_{i=0}^{N} \frac{\langle g_i, x_\star \rangle^2}{\|g_i\|^2}.
\end{align*}
Moreover, we note that the optimality conditions for the inner maximization imply that $\lambda_{i,j} = 0$ for $i > j$. Solving the inner maximization yields~\eqref{upper_triangular_dual_pep} after the change of variables  $\nu = s/2$ and $t_{i,j} = \lambda_{i,j}/L(\delta_{i,j})$.


\end{proof}

\begin{proof}[\bf Proof of \cref{thm:UOGM-k/delta-q}]
    Define the following quantity for $n=0,\dots,N$,
    \begin{align*}H_{n} &= \tau_{n}\left(f_\star-f_{n}+\frac{1_{n < N}}{2L(\delta_{n,n+1})}\|g_{n}\|^2 \right)+\frac{1}{2}\|x_0-x_\star\|^2-\frac{1}{2}\|z_{n+1}-x_\star\|^2\\& \qquad+\sum_{i=1}^{n}\tau_{i-1} \delta_{i-1,i}+\sum_{i=0}^{n}(\tau_i-\tau_{i-1}) \delta_{\star,i}.\end{align*}
    We claim that Algorithm~\ref{alg:IS-OGM} inductively maintains the nonnegativity of $H_n$.

    Recall $\mathcal{Q}_{i,j,\delta_{i,j}} := f_i-f_j - \inner{g_j}{x_i-x_j}-\frac{1}{2L(\delta_{i,j})}\|g_i-g_j\|^2+\delta_{i,j}$. With $\tau_{-1} := 0$ and for any $x_0$ and our choice of $\tau_0 = \frac{1}{L(\delta_{0,1})}+\frac{1}{L(\delta_{\star,0})}$, setting $z_1 := x_0-\tau_0g_0$, the base case holds that $H_0 = \tau_0 \mathcal{Q}_{\star,0,\delta_{\star,0}} \geq 0$. For any $n=1,\dots,N$, this nonnegativity is inductively maintained by observing the identity
    $$H_n =H_{n-1}+\tau_{n-1}\mathcal{Q}_{n-1,n,\delta_{n-1,n}}+(\tau_n-\tau_{n-1})\mathcal{Q}_{\star,n,\delta_{\star,n}} $$
    which shows $H_n$ is a sum of three nonnegative quantities (and hence is nonnegative). This identity follows from the carefully constructed choices of $\tau_n, x_n, z_n$\footnote{As a verification of an effective equivalent identity, see, for example,~\cite[Lemma 3]{grimmer2026minimaxoptimalitysubgameperfect}.}.

    Hence $H_N\geq 0$. Rearranging the definition of $H_N$, this establishes a convergence guarantee of
    \begin{equation}\label{general_UOGM_rate}
        f_N-f_\star \leq \frac{\frac{1}{2}\|x_0-x_\star\|^2 +\sigma_N}{\tau_N}
    \end{equation}
    where $\sigma_N := \sum_{i=1}^N \tau_{i-1}\delta_{i-1,i} +\sum_{i=0}^N (\tau_i-\tau_{i-1})\delta_{\star,i}$. The remainder of the proof follows from analyzing the sequence $\tau_n$, which is entirely determined by our choice of tolerances $\delta_{i-1,i}$ for $i \leq n$. Below, we carry out these final calculations. Given these, the guarantee outlined for the $(\beta,p)$-\Holder smooth setting is a direct result from the $L(\delta)=\kappa/\delta^q$-inexactly smooth guarantee, the implications of \cref{prop:Equiv Char of HS functions}, and the substitutions $p = \frac{1-q}{1+q}$ and $\beta = \left(\kappa\left(\frac{2}{q}\right)^{q}\right)^{\frac{1}{\left(1+q\right)}}$.
      
    For the chosen, optimized tolerances with
\begin{equation}\label{eq:c_0}
 \delta_{n-1,n} = \left(\frac{q\kappa D^2}{(q+1)^2(N+1)}\right)^\frac{1}{q+1}{n^{-\frac{2}{q+1}}}, \qquad
    \delta_{\star, n} = 0.
\end{equation}
the recurrence defining $\tau_n$ collapses to
\begin{equation}\label{eq:tau-recurrence-simplified}
(\tau_n-\tau_{n-1})^2=\frac{\tau_n}{L(\delta_{n,n+1})}+\frac{\tau_{n-1}}{L(\delta_{n-1,n})}\quad \forall n < N,\quad (\tau_N-\tau_{N-1})^2=\frac{\tau_{N-1}}{L(\delta_{N-1,N})}.
\end{equation}
Since $\delta_{n-1,n}$ is decreasing in $n$, $L(\delta_{n-1,n})$ is nondecreasing in $n$, so for $n<N$,
$$
\frac{\tau_n}{L(\delta_{n,n+1})}+\frac{\tau_{n-1}}{L(\delta_{n-1,n})} \geq \frac{\tau_n+\tau_{n-1}}{L(\delta_{n,n+1})}\geq \frac{\left(\sqrt{\tau_n}+\sqrt{\tau_{n-1}}\right)^2}{2L(\delta_{n,n+1})},
$$
using $a^2+b^2\ge\frac{(a+b)^2}{2}$. Taking the square root in \eqref{eq:tau-recurrence-simplified} and dividing by $\sqrt{\tau_n}+\sqrt{\tau_{n-1}}$ yields
\begin{equation}\label{eq:sqrt-tau-lower}
\sqrt{\tau_n}-\sqrt{\tau_{n-1}}\geq \frac{1}{\sqrt{2L(\delta_{n,n+1})}}.
\end{equation}

This bound is asymptotically sharp. Bounding
\eqref{eq:tau-recurrence-simplified} for $n < N$ from above instead gives
$$
\frac{\tau_n}{L(\delta_{n,n+1})}+\frac{\tau_{n-1}}{L(\delta_{n-1,n})}
\leq \frac{\tau_n+\tau_{n-1}}{L(\delta_{n-1,n})}
\leq\frac{\left(\sqrt{\tau_n}+\sqrt{\tau_{n-1}}\right)^2}{L(\delta_{n-1,n})},
$$
using $L(\delta_{n,n+1})\ge L(\delta_{n-1,n})$ and
$\tau_n+\tau_{n-1}\le(\sqrt{\tau_n}+\sqrt{\tau_{n-1}})^2$. Taking the square root
and dividing by $\sqrt{\tau_n}+\sqrt{\tau_{n-1}}$ gives the companion to
\eqref{eq:sqrt-tau-lower}: $\sqrt{\tau_n}-\sqrt{\tau_{n-1}}\leq\frac{1}{\sqrt{L(\delta_{n-1,n})}}$ for $n < N$.
Substituting \eqref{eq:c_0} into $L(\delta)=\kappa/\delta^q$ and dividing by $\sqrt{\tau_{n-1}} \geq \sum_{i=1}^{n-1} \frac{1}{\sqrt{2L(\delta_{i,i+1})}}$ gives
$$
0\leq\frac{\sqrt{\tau_n}-\sqrt{\tau_{n-1}}}{\sqrt{\tau_{n-1}}}
\leq\frac{\sqrt{2}n^{-\frac{q}{q+1}}}{\sum_{i=1}^{n-1}(i+1)^{-\frac{q}{q+1}}}
=\frac{\sqrt{2}}{q+1}n^{-1}\left(1+o(1)\right)
\to 0,
$$
where the equality uses
$\sum_{i=1}^{n}(i+1)^{-\frac{q}{q+1}}=(q+1)(n+1)^{\frac{1}{q+1}}(1+o(1))$.
Therefore, $\tau_n/\tau_{n-1}\to1$ uniformly in $N$. Since
$L(\delta_{n,n+1})/L(\delta_{n-1,n})=(1+1/n)^{\frac{2q}{q+1}}\to1$ as well,
\eqref{eq:tau-recurrence-simplified} reads
$(\tau_n-\tau_{n-1})^2=\frac{2\tau_n}{L(\delta_{n,n+1})}(1+o(1))$, and with
$\sqrt{\tau_n}+\sqrt{\tau_{n-1}}=2\sqrt{\tau_n}(1+o(1))$ this yields the asymptotic
equality $\sqrt{\tau_n}-\sqrt{\tau_{n-1}}=\frac{1}{\sqrt{2L(\delta_{n,n+1})}}(1+o(1))$
for $n<N$. Telescoping these terms gives 
\begin{equation}\label{eq:tau-asymp}
\tau_n=\frac{(q+1)^\frac{2}{q+1}\left(q D^2\right)^{\frac{q}{q+1}}}{2\kappa^\frac{1}{q+1}} \, \frac{(n+1)^{\frac{2}{q+1}}}{(N+1)^{\frac{q}{q+1}}}\left(1+o(1)\right), \quad \forall n \leq N,
\end{equation}
where the modified case of $n=N$ only adds lower order terms.
Considering our choices of $\delta_{n-1,n}$, it holds $\tau_{n-1}\delta_{n-1,n} = \frac{qD^2}{2(N+1)}(1+o(1))$ for each $n \leq N$ so
\begin{equation}\label{eq:sigma-asymp}
\sigma_N = \sum_{i=1}^N \tau_{i-1}\delta_{i-1,i} =\frac{ qD^2}{2}\left(1+o(1)\right).
\end{equation}
Substituting \eqref{eq:tau-asymp} and \eqref{eq:sigma-asymp} at $n=N$ into \eqref{general_UOGM_rate} gives the claimed result.
\end{proof}

\end{document}